\documentclass[a4paper, 12pt, pdftex,reqno
]{amsart}
\usepackage[hang, flushmargin, bottom]{footmisc} 

\usepackage[hyperindex,breaklinks,pdftex, bookmarks]{hyperref}

\usepackage[all,  curve, frame, arc, color]{xy}

\usepackage{mathtools}

\usepackage{helvet}         
\usepackage{courier}        
\usepackage{graphicx}        
\usepackage{multicol}        

\makeindex             

\usepackage[utf8x]{inputenc}
\usepackage[english]{babel}
\usepackage{amsmath} 
\usepackage{amscd}
\usepackage{amsfonts}
\usepackage{amssymb}
\usepackage{amsthm} 
\usepackage{mathabx}
\usepackage{mathrsfs}
\usepackage{graphicx}
\usepackage{textcomp}
\usepackage{url}
\usepackage{amsopn,latexsym,amscd}
\usepackage{color}
\usepackage[disable]{todonotes}
\usepackage{fullpage}
\usepackage{graphicx}

\makeatletter
\newcommand{\eqnum}{\refstepcounter{equation}\textup{\tagform@{\theequation}}}
\makeatother

\newtheorem{thm}{Theorem}[subsection]
\newtheorem{lem}[thm]{Lemma}  

\newtheorem{prop}[thm]{Proposition}
\newtheorem{cor}[thm]{Corollary}

\newtheorem{question}[thm]{Question}

\newtheorem{Rthm}{Theorem}

\theoremstyle{definition}
\newtheorem{df}[thm]{Definition}
\newtheorem{scholium}[thm]{Scholium}
\newtheorem{notation}[thm]{Notation}
\newtheorem{es}[thm]{Example}
\theoremstyle{remark}
\newtheorem{rem}[thm]{Remark}

\DeclareFontFamily{U}{matha}{\hyphenchar\font45}
\DeclareFontShape{U}{matha}{m}{n}{
      <5> <6> <7> <8> <9> <10> gen * matha
      <10.95> matha10 <12> <14.4> <17.28> <20.74> <24.88> matha12
      }{}
\DeclareSymbolFont{matha}{U}{matha}{m}{n}
\DeclareFontSubstitution{U}{matha}{m}{n}
\DeclareMathSymbol{\cUp}{\mathbin}{matha}{'131} 

\newcommand{\Z}{\mathbb{Z}}
\newcommand{\C}{{\mathbb{C}}}
\newcommand{\R}{{\mathbb{R}}}
\newcommand{\Q}{{\mathbb{Q}}}

\newcommand{\tops}{\texttt{Top}}
\newcommand{\scwols}{\texttt{Scwol}}

\newcommand{\B}{{\mathcal{B}}}

\newcommand{\MS}{{\mathcal{S}}}

\newcommand{\Pc}{{\mathcal{P}}}
\newcommand{\Ic}{{\mathcal{I}}}

\newcommand{\Kc}{{\mathcal{K}}}

\newcommand{\Sc}{{\mathcal{S}}}
\newcommand{\Cc}{{\mathcal{C}}} 

\newcommand{\Fc}{{\mathcal{F}}}
\newcommand{\Tc}{{\mathcal{T}}}

\newcommand{\Ll}{{\mathcal{L}}}

\newcommand{\oal}{{\overline{\alpha}}}

\newcommand{\tal}{{\widetilde{\alpha}}}

\newcommand{\oP}{\overline{P}}
\newcommand{\tga}{\widetilde{\gamma}}
\newcommand{\oga}{\overline{\gamma}}
\newcommand{\ode}{\overline{\delta}}
\newcommand{\tde}{\widetilde{\delta}}

\def\HH{\mathcal H}
\def\MM{\mathcal M}
\newcommand{\iF}[2]{i_{ #1 }( #2 )}

\newcommand{\rk}{{\operatorname{rk}}}
\newcommand{\codim}{{\operatorname{codim}}}

\newcommand{\GC}{\textstyle \int}

\newcommand{\Sal}{\operatorname{Sal}}
\newcommand{\wL}{\widetilde{L}}
\newcommand{\into}[0]{\hookrightarrow}
\newcommand{\id}{\mathrm{id}}
\newcommand{\A}{{\mathcal{A}}} 
\newcommand{\Nc}{{\mathcal{N}}} 
\newcommand{\Ah}{{\mathcal{A}}} 
\newcommand{\mI}{\mathcal{J}} 
\newcommand{\Ar}[1]{\A[#1]} 
\newcommand{\Sl}[1]{\mathcal{S}_{#1}} 
\newcommand{\gre}[1]{\vert #1 \vert} 
\newcommand{\Aa}{\operatorname{A}} 
\newcommand{\pra}{
           \mathrel{\raisebox{.1em}{%
           \reflectbox{\rotatebox[origin=c]{90}{$\triangleright $}}}}}
\newcommand{\Lc}[1]{{\mathscr{L}}_{#1}}
\newcommand{\Ab}{\operatorname{B}} 
\newcommand{\colim}{\operatorname{colim}}
\newcommand{\hocolim}{\operatorname{hocolim}}
\newcommand{\Hom}{\operatorname{Hom}}

\newcommand{\SF}[1]{\Sc^{#1}}
\newcommand{\sv}[1]{\gamma_{#1}}

\newcommand{\mk}[1]{
\stackrel{m_{#1}}{\longrightarrow}
}

\newcommand{\colobj}[2]{
\left[\begin{smallmatrix} #1 \\ #2 \end{smallmatrix}\right]
}
\newcommand{\colmor}[2]{
\xrightarrow{
\left[
  \begin{smallmatrix}
    #1 \\ #2
  \end{smallmatrix}
\right]
}
}

\begin{document}

\title[Short title]{The integer cohomology algebra\\ of toric arrangements}
\author{Filippo Callegaro and Emanuele Delucchi}
\date{\today 
}
\address[F. Callegaro]{
Dipartimento di Matematica, University of Pisa, Italy.
}
\email{callegaro@dm.unipi.it}

\address[E. Delucchi]{Departement de math\'ematiques,
  Universit\'e de Fribourg, Switzerland.} 
\email{emanuele.delucchi@unifr.ch}

\begin{abstract}
We compute the cohomology ring of the complement of a  
toric arrangement with integer coefficients and investigate its
dependency from the arrangement's combinatorial data.
To this end, we study a morphism of spectral sequences associated to
certain combinatorially defined subcomplexes of the toric Salvetti
category in the complexified case, and use a technical argument in
order to extend the results to full generality. As a byproduct we obtain:

\begin{itemize}
\item[-] a ``combinatorial'' version of Brieskorn's lemma in terms of
Salvetti complexes of complexified arrangements,
\item[-] a uniqueness result for realizations of arithmetic matroids
with at least one basis of multiplicity 1.
\end{itemize}
\end{abstract}

\maketitle

\section{Introduction}

The goal of this paper is to give a presentation of the cohomology
ring with integer coefficients of the complement of a toric
arrangement -- i.e., of a family of level sets of characters of the
complex torus -- and to investigate its dependency from the poset of
layers of the arrangement.

This line of research can be traced back to Deligne's seminal work on complements
of normal crossing divisors in smooth projective varieties
\cite{Deligne} and has been extensively and successfully carried out
in the case of  arrangements of hyperplanes in complex space, where
the integer cohomology ring of the complement is a well-studied
object with strong combinatorial structure. In particular, it can be defined purely in terms of the
intersection poset of the arrangement, and in greater generality, for
any matroid, giving rise to the class of so-called {\em Orlik-Solomon
  algebras}. We refer to Yuzvinsky's survey \cite{Yuzvinsky} for a
thorough introduction and a
``tour d'horizon'' of the range of directions of study focusing on 
OS-algebras.

Recently, the study of hyperplane arrangements has been taken as a
stepping stone towards different kinds of generalizations. Among
these let us mention the work of Dupont \cite{dupont2013} developing
algebraic models for complements of divisors with hyperplane-like
crossings and of Bibby \cite{bibby2013} studying the rational
cohomology of complements of arrangements in abelian varieties. Both
apply indeed to the case of interest to us, that of toric
arrangements.

Besides being a natural step beyond arrangements of hyperplanes in the
study of complements of divisors, our motivation for considering toric
arrangements stems also from recent work of De Concini, Procesi and
Vergne which puts topological and combinatorial properties of toric
arrangements in a much wider context (see \cite{DPV} or the book
\cite{dp2005}) and spurred a considerable amount of research aimed at
establishing a suitable combinatorial framework. This research was
tackled along two main directions.
 
One such direction, from algebraic combinatorics,
led Moci \cite{Moci} to introduce a suitable generalization of the Tutte
polynomials and then, jointly with d'Adderio \cite{dAMo}, to the development of
arithmetic matroids (for an up-to date account see Br\"and\'en and
Moci \cite{BrMo}). These objects, as well as others like matroids
over rings \cite{FiMo}, exhibit an interesting structure theory and
recover earlier enumerative results by Ehrenborg, Readdy and Slone
\cite{ERS} and Lawrence \cite{Law} but, as
of yet, only bear an enumerative relationship with topological or
geometric invariants of toric arrangements - in particular, these
structures do not characterize their intersection pattern (one attempt towards closing this gap has been made by
considering group actions on semimatroids \cite{DeRi}). 

The second direction is the study of the combinatorial invariants of
the topology and geometry of toric arrangements: our work is a
contribution in this direction, and therefore we now briefly review earlier
contributions. The Betti numbers of the complement to a toric
arrangement were known at least since work of Looijenga
\cite{looi93}. De Concini and Procesi \cite{dp2005} related these
Betti numbers to the combinatorics of the poset of connected
components of intersections in the context of their computation of a
presentation of the cohomology ring over $\mathbb C$ for unimodular
arrangements (i.e., those arising from kernels of a totally unimodular
set of characters), from which they also deduce formality for these
arrangements. A first combinatorial model for the homotopy type of
complements of toric arrangements was introduced by Moci and
Settepanella \cite{MociSettepanella} for ``centered'' arrangements
(i.e., defined by
kernels of characters) which induce a regular CW-decomposition of the
compact torus $(S^1)^d \subseteq (\mathbb C^*)^d$, and was
subsequently generalized to the case of ``complexified'' toric
arrangements ($S^1$-level sets of characters) by d'Antonio and the
second author \cite{dD2010} who, on this basis, also gave a
presentation of the complement's fundamental group. In later work \cite{dD2013},
d'Antonio and the second author also proved that complements of
complexified toric arrangements are minimal spaces (i.e., they have
the homotopy of a CW-complex where the $i$-dimensional cells are
counted by the $i$-th Betti number): in particular, the integer
cohomology groups are torsion-free and are thus determined by the
associated arithmetic matroid. This raises the question of whether, as
is the case with the OS-Algebra of hyperplane arrangements, the integer
cohomology ring is combinatorially determined. The work of Dupont
\cite{dupont2013} and Bibby \cite{bibby2013} mentioned earlier,
although more general in scope, does
include the case of toric arrangements but falls slightly short of our
aim in that on the one hand it uses field coefficients\footnote{A recent private conversation
with Cl\'ement Dupont indicated that at least parts of his methods
could be generalized to integer coefficients.} and on the other
hand computes only the bigraded module associated to a filtration of the cohomology algebra obtained as the abutment of a spectral sequence.

Lately, Deshpande and Sutar \cite{Deshpande}, by an explicit study of the
Gysin sequence, gave a
sufficient criterion for the complex cohomology algebra of a toric arrangement to be
generated in first degree and to be formal.

In this paper we pair the (by now standard) spectral sequence argument
with a very explicit combinatorial analysis of the toric Salvetti
complex and can thus compute the full cohomology algebra over
the integers of general complexified toric arrangements. The
generalization to non-complexified case relies then on a technical argument.
We give two
presentations of the cohomology algebra and discuss its dependency from the poset of
connected components of intersections. In
the case of arrangements defined by kernels of characters there is
also an associated arithmetic matroid and in this case we prove that
when the defining set of characters contains an unimodular basis the
arithmetic matroids determines the integer cohomology algebra. The
precise results will be stated in Section \ref{sec:overview}, together with
a brief survey of the
architecture of the remainder of our work.

\noindent {\bf Acknowledgements.} This work was started during a Research in pair stay at the CIRM-FBK,
Trento, Italy in March 2014. Support by CIRM-FBK is gratefully
acknowledged. Emanuele Delucchi has been partially supported by the
Swiss National Science Foundation professorship grant PP00P2\_150552/1.

\setcounter{tocdepth}{1}
\tableofcontents

\section{Overview and statement of results}\label{sec:overview}

\subsection{Main definitions}\label{sec:maindef}
Let $T = (\C^*)^d$ be the complex torus and let $T_c = (S^1)^d$ be the compact subtorus of $T$.

A {\em toric arrangement} is a finite set

$$\A = \{Y_1, \cdots, Y_{n} \}$$ 
where, for every $i=1,\ldots,n$,
$$Y_i:=\chi_i^{-1}(a_i) 
$$
with $\chi_i\in \Hom(T,\mathbb C^*)$ and $a_i\in \mathbb C^*$. The arrangement $\A$ is called
 {\em complexified} if $a_i\in S^1$ for every $i$.

A \emph{layer} of $\A$ is a connected component of a non-empty intersections of elements of $\A$. 
The \emph{rank} of a layer $L$, is its codimension as a complex submanifold in $T$.
We order layers by reverse inclusion: $L \leq L'$ if $L' \subseteq L$.
Let $\Cc$ be the poset of layers associated to $\A$ and let $\Cc_q$ be the subset 
of $\Cc$ given by the layers $L \in \Cc$ with $\rk (L) =q$.

The complement of a toric arrangement $\A$ is the space
$$
M(\A):= T\setminus \bigcup \A.
$$

\begin{rem} A toric arrangement is called essential if the layers of
  minimal dimension have dimension $0$ (equivalently, the rank of
  $\mathcal C(\A)$ as a poset equals the dimension of $T$). Notice that for any
  nonessential toric arrangement $\A$ there is an essential toric
  arrangement $\A'$ with $M(\A)= (\mathbb C^*)^{r}\times M(\A')$,
  where with $r=\rk(\mathcal C(\A))$, see
 \cite[Remark 4]{dD2010}.
\end{rem}

As in the case of an hyperplane arrangement, we define the rank of a toric arrangement
$\rk(\A):=\rk(\mathcal C(\A))$.

\def\Aup{\A^\upharpoonright}
To every toric arrangement $\A$ corresponds a periodic affine
hyperplane arrangement $\Aup$ in the universal cover $\C^d$ of the
complex torus. The hyperplane arrangement $\Aup$ is complexified
exactly when $\A$ is. 

\begin{df}\label{df:A0}
  For a toric arrangement $\A$  define
  the hyperplane
  arrangement $$\A_0:=\{Y^\upharpoonright_1,\ldots,Y_n^\upharpoonright\}$$
  where, for $i=1,\ldots ,n$, $Y^\upharpoonright_i$ is the translate at the origin of any
  hyperplane of $\Aup$ lifting $Y_i$.

  Given a layer $L\in \Cc(\A)$, define then 
  $$
  \Ar{L}:=\{Y_j^\upharpoonright\in\A_0 \mid L\subseteq Y_j\}
  $$
\end{df}

\begin{rem} \label{rem:posetArL}
  It is immediate to see that the intersection lattice of the
  hyperplane arrangement $\Ar{L}$ is poset-isomorphic to $\Cc_{\leq L}$.
\end{rem}

\subsection{Background on hyperplane arrangements} \label{sec:backarr}

The fact that the cohomology ring of an arrangement's complement is
combinatorial can be made precise as follows.

Let $\Ah$ be an arrangement of hyperplanes in $\C^d$. The main
combinatorial invariant of $\Ah$ is the poset
$$\Ll(\Ah):=\{\cap \Kc \mid \Kc\subseteq \Ah\}$$
partially ordered by reverse inclusion: $X\geq Y$ if $X\subseteq
Y$. Notice that $\Ll$ contains a unique minimal element that we call
$\hat 0$, corresponding to the intersection over the empty set.
When $\Ah$ is central (i.e. $\cap \Ah \neq \emptyset$), this poset is a
geometric lattice and thus defines a (simple) matroid associated to
the arrangement.

The $j$-th Betti number of the complement $M(\Ah):=\C^d\setminus \cup
\Ah$ can be stated in terms of $\Ll$ as
$$
\beta_j(M(\Ah)) = \sum_{x\in \Ll_j} \mu_\Ll(\hat 0,x)
$$
where $\mu_\Ll$ denotes the M\"obius function of $\Ll$ and $\Ll_j$ is
the set of elements of $\Ll$ of rank $j$.

Brieskorn \cite{Brie1973} proved that the cohomology of $M(\Ah)$ is
torsion-free, thus the additive structure of
$H^*(M(\Ah);\Z)$ is determined by $\Ll$. Moreover, we have the
following fundamental result expressing the cohomology of $\Ah$ in
terms of the top cohomology of subarrangements of the form
\begin{equation}
  \label{eq:1}
  \Ah_X :=\{H\in \A \mid X\subseteq H\} \quad\quad \textrm{ for }X\in \Ll(\Ah).
\end{equation}
\begin{lem}[Brieskorn Lemma \cite{OT}] \label{lem:brieskorn} 
  Let $\Ah$ be an arrangement of hyperplanes. For all $k$ the map
$$
\bigoplus_{X\in \Ll, \rk(X) = k} H^k(M(\Ah_X),\Z) \to H^k(M(\Ah);\Z)
$$
induced by the inclusions $M(\Ah) \into M(\Ah_X)$ is an isomorphism of
groups.
\end{lem}

\begin{df}\label{def:b}
Given a hyperplane arrangement $\Ah$ and an intersection $X\in \Ll$,
we will denote by $b^k: H^k(M(\Ah_X); \Z) \to H^k(M(\Ah);\Z)$ 
the map
given by inclusion into the $X$-summand in the decomposition given in
Brieskorn's Lemma.
\end{df}

As far as the algebra structure is concerned, Orlik and Solomon
defined an abstract algebra in terms of the matroid associated to
$\Ah$, then proved it isomorphic to the cohomology algebra using
induction on rank via the
{\em deletion-restriction recurrence}, i.e., the exact sequence

\begin{equation}
  \label{eq:2}
  0\to H^k(M(\Ah '); \Z) \to H^k(M(\Ah);\Z) \to H^{k-1}(M(\Ah'');\Z)
  \to 0
\end{equation}
 valid for all $k>0$, which, given any $H_0\in \Ah$, connects the
 cohomology of the complement of the {\em deleted} arrangement $\Ah':= \Ah\setminus
 \{H_0\}$ and the cohomology of the complement of 
 the {\em restricted} arrangement $\Ah'':=\{H\cap H_0 \mid H\in \Ah' \}$.

The abstract presentation given by Orlik and Solomon is the following.

\begin{df}[Orlik-Solomon algebra of a hyperplane arrangement]
  Consider a central arrangement of hyperplanes $\Ah=\{H_1,\ldots ,
  H_n\}$ and let $E^*$ denote the graded exterior algebra generated by $n$ elements
  $e_1,\ldots,e_n$ in degree $1$ over the ring of integers.
  Define an ideal $\mI(\Ah)$ as generated by the set:
$$
\big\{\partial e_X \,\, \big\vert \,\, X\subseteq [n]; \,\, \codim \bigcap_{i\in
X} H_i < \vert X \vert\big\}
$$
where, for $X=\{i_1,\ldots,i_k\}\in [n]$, we write $\partial e_X :=
e_{i_1}\cdots e_{i_k}$ and define $$\partial e_X = \sum_{j=1}^k
(-1)^{j-1} e_{X\setminus \{i_j\}}.$$

The Orlik-Solomon algebra of $\Ah$ is then defined as the quotient 
$$
\operatorname{OS}^*(\Ah) := E^*/\mI(\Ah).
$$

\end{df}


\begin{thm}[Orlik and Solomon \cite{OS}]
  For every central arrangement of hyperplanes $\Ah$, there is an isomorphism of graded algebras
$$
\operatorname{OS}^*(\Ah) \simeq H^*(M(\Ah);\Z)
$$
\end{thm}

\subsection{Results}

We now briefly formulate our main results. The remainder of the paper
will be then devoted to the proofs. 
Let us 
consider a
 toric arrangement $\A$,  writing $\Cc$ for the poset of layers of $\A$.

\subsubsection{The algebra $\Aa(\A)$}
\begin{df} \label{df:diagram} 
Let $\Lc{}: \Cc \to \{\Z\mbox{-algebras} \}$ be the diagram defined by 
$$
L \mapsto \Lc{L}:= H^*(L;\Z) \otimes H^*(M(\Ar{L});\Z)
$$
and 
$$
L'\leq L \mapsto \Lc{L'\leq L} = i^* \otimes b : \Lc{L'} \to \Lc{L},
$$
where 
$$
i^*: H^*(L';\Z) \to H^*(L;\Z) 
$$
is the natural morphism induced by the inclusion $L \stackrel{i}{\into} L'$ 
and
$b$ denotes the map of Definition \ref{def:b}.
\end{df}

The algebra $\Lc{L}$ can be graded with the graduation induced by  
$H^*(M(\Ar{L});\Z)$, 
hence we have
$$
\Lc{L}^q= H^*(L;\Z) \otimes H^q(M(\Ar{L});\Z).
$$

\begin{df} \label{def:algebra_a}
We define the 
algebra $\Aa(\A)$ as the direct sum
$$
\bigoplus_{L\in \Cc} \Lc{L}^{\rk (L)}
$$
with multiplication map defined as follows. Let $L,L' \in \Cc$ be two layers. Consider two classes $\alpha \in \Lc{L}^{\rk(L)}$ and $\alpha' \in \Lc{L'}^{\rk(L')}$. We define the product
$$
((\alpha) \pra (\alpha'))_{L''}\!:= \left\{ 
\begin{array}{ll}
\!\Lc{L \leq L''}(\alpha) \cUp \Lc{L' \leq L''}(\alpha') & \mbox{if } L\! \cap\! L'\! \leq\! L'' \mbox{ and } \rk (L'') = \rk (L) + \rk (L'); \\
\!0 & \mbox{otherwise.} 
\end{array}
\right.
$$

\end{df}

\begin{Rthm} \label{thm:main}
  There is an isomorphism of algebras 
$$H^*(M(\A);\Z) \simeq \Aa(\A).$$
\end{Rthm}

\subsubsection{The algebra $\Ab(\A)$}

\begin{df} \label{def:coherent}
Let $\alpha$ be an element in the direct sum $\bigoplus_{L \in \Cc} \Lc{L}$.
We say that $\alpha$ is \emph{coherent} if 
for every integer $q$ and for every $L \in \Cc_{>q}$ 
we have that
$$
\sum_{L' \in (\Cc_{\leq L})_{q}}\Lc{L' \leq L}(\alpha_{L'}^q) = \alpha_{L}^q
$$
where $\alpha_{L}^q$ (resp. $\alpha_{L'}^q$) is the component of $\alpha_{L}$ (resp. $\alpha_{L'}$) in 
$\Lc{L}^q$ (resp. $\Lc{L'}^q$).
\end{df}

Coherent elements in $\bigoplus_{L \in \Cc} \Lc{L}^q$
generates a subgroup, in fact they form a subalgebra (see Proposition
\ref{prop:colim}) that we call $\Ab(\A)$ (see Definition \ref{def:algebra_b}).

\begin{Rthm}[see Proposition \ref{prop:iso_algebras}] \label{thm:iso_alg}
  The algebras $\Aa(\A)$ and $\Ab(\A)$ are isomorphic.
\end{Rthm}

\subsubsection{Combinatorial aspects} \label{sec:combin_aspects}
After obtaining a grasp on the cohomology algebra it is natural,
especially in comparison with the case of hyperplane arrangements, to
ask the question of whether (and in what sense) it is combinatorially
determined. The most natural combinatorial structure to consider in
this context is of course the poset of layers $\Cc$, both because this
is the direct counterpart of the intersection poset of a hyperplane
arrangement and because we already know it determines the Betti
numbers and hence (by torsion-freeness) the cohomology
groups. As an additional element of similarity, we prove in \S \ref{sec:Whitney} that just as in the case of hyperplane
arrangements the cohomology groups can be obtained as the Whitney
homology of $\Cc$.  When the arrangement is centered (i.e., defined by kernels of characters), another associated structure is the
arithmetic matroid of the defining characters \cite{BrMo}. While for
hyperplane arrangements the two counterparts -- (semi)lattice of flats
and (semi)matroid -- are equivalent
combinatorial structures, in our situation it is still true that in
the centered case $\Cc$ determines an
arithmetic matroid, but it is not known at present how to construct $\Cc$
from an abstract arithmetic matroid. Thus the question is the following.

\begin{question}\label{qu:combi}
Is the isomorphism type of the integer cohomology ring of the complement of a complexified toric arrangement determined by the poset of layers?
\end{question}

The strongest affirmative result we can prove at the moment is that
for centered toric arrangements which possess a unimodular basis the
poset $\Cc$ does determine the cohomology algebra. Indeed, in this
case the arithmetic matroid
determines the arrangement itself: our Theorem \ref{thm:ricostruzione} shows
that if an arithmetic matroid with a unimodular basis is
representable, then the representation is unique up to sign reversal
of the vectors.

We cannot at this moment solve Question \ref{qu:combi} in the general
(non-centered, without unimodular bases) case, and will close our work
with an example that we hope will illustrate some of the delicacy of
the situation, namely: even if two cohomology rings are isomorphic,
there needs not be a ``natural'' isomorphism.

 \begin{rem}
In the following sections we will consider only complexified toric
arrangements. The extension of our results to general,
non-complexified toric arrangement will be given in Section
\ref{sec:general}.
 \end{rem}
 
\subsection{Structure of the paper}
Given a complexified toric arrangement $\A$ (defined in \S \ref{sec:maindef}), our combinatorial model for the homotopy type of the
complement of $\A$ is the toric
Salvetti complex $\Sal(\A)$, in the formulation  given in \cite{dD2013}, in particular as the nerve of
an acyclic category obtained as 
homotopy colimit of a diagram of posets. 
In
Section \ref{sec:preparation} we review some basic facts about the combinatorics and topology of
acyclic categories and establish some facts about the combinatorial
topology of Salvetti complexes of complexified hyperplane
arrangements. In particular, 
\begin{itemize}
	\item[(a)] we identify maps between poset of cells of Salvetti
	complexes which induce the Brieskorn isomorphisms (Proposition
	\ref{cor:cbl}, which we call a ``combinatorial
	Brie\-skorn Lemma'' for complexified arrangements).
\end{itemize}
The next step is carried out in Section \ref{sec:tsc}, where
\begin{itemize}
	\item[(b)] for every connected component $L$ of an intersection of
	elements of $\A$ we define a subcomplex
	\begin{equation}
	\Sc_{L}\hookrightarrow
	\Sal(\A)\label{eq:5}
	\end{equation}
	with the homotopy type of the product $L\times M(\Ar{L})$,
	where $\Ar{L}$ is the arrangement of hyperplanes in $\mathbb C^d$ defined by $\A$ in
	the tangent space to $(\mathbb C^*)^d$ at any generic point in $L$
	and $M(\Ar{L}):=\mathbb C^d \setminus
	\bigcup \Ar{L}$.
	\item[(c)] Moreover, using (a) we can identify, and study at the level of cell
	complexes, the maps that are induced in cohomology by the inclusions
	\eqref{eq:5} and between $H^*(\Sc_{L})$ and $H^*(\Sc_{L'})$ for
	$L\subseteq  L'$.
\end{itemize}

Section \ref{sec:SpSe} is devoted to the inspection of the spectral
sequence $\widehat E_{r}^{p,q}$ for $\Sal(\A)$ coming from the formulation of
the toric Salvetti complex as a homotopy colimit (see Segal \cite{segal68}) (which is
indeed equivalent to the Leray spectral sequence of the inclusion of
$M(\A)$ into the torus) and the (trivial) spectral sequences ${}_L E_{r}^{p,q}$
for $\Sc_L$ coming from projection on the torus factor.
These spectral sequences all degenerate at the second page. 
\begin{itemize}
	\item[(d)] The map of
	spectral sequences induced
	by the inclusions \eqref{eq:5} leads us to consider the following
	commuting diagram (of groups).
	$$
	\begin{CD}
	H^*(\Sal(A)) @>>> H^*(\coprod_L \Sc_L) \\
	@VVV @VVV\\
	\widehat E_{2}^{p,q} @>>> \bigoplus_L \,\,{}_L E_{2}^{p,q} 
	\end{CD}\begin{CD}
	=\bigoplus_L H^*(L)
	\otimes H^*(M(\Ar{L}))\\
	@.\\
	@.
	\end{CD}
	$$
	After some preparation in Section \ref{ss:algebras}, the gist of our proof is reached in Section \ref{ss:proof_main}, where we use the (explicit) bottom map
	(of groups) to prove injectivity and to characterize, via (c), the
	image of the top map (of rings). We do this by presenting the image as
	an algebra $A(\A)$ obtained by defining ''the natural product'' on
	$\bigoplus_L H^*(L)\otimes H^{\codim L} (M(\Ar{L}))$ (Definition \ref{def:algebra_a}) as well as an
	algebra $B(\A)$ of ``coherent elements'' of $\bigoplus_L
	H^*(L)\otimes H^{*} (M(\Ar{L}))$ (Definition \ref{def:coherent}).
\end{itemize}

In Section \ref{sec:general} we extend our results to general
(non-complexified) toric arrangements, using a deletion-restriction
type argument which allows us to reduce to the complexified case. 
We then close with Section \ref{sec:combiforse} where we investigate
the dependency of the cohomology ring structure from the poset $\Cc$ of
connected components of intersections, trying to identify similarities
and differences with the case of hyperplane arrangements, where this cohomology
structure is completely determined by the poset of intersections. We
will show that the cohomology groups are, as in the hyperplane case,
obtained as Whitney homology of the intersection poset (\S \ref{sec:Whitney}), and we prove that
$\Cc$ determines the cohomology ring of every toric arrangement which is defined as the set of kernels of a
family of characters which contains at least one unimodular basis
(Theorem \ref{thm:ricostruzione}). We conclude by giving two examples (\S \ref{sec:examples}) 
which illustrate the subtle relationship of the combinatorics of the
poset of layers with the ring structure of the cohomology. 
First, we present two 
(centered) arrangements in $(\mathbb C^*)^2$ with
isomorphic posets of layers which do indeed have isomorphic cohomology
rings, but no `natural' isomorphism exists (i.e., no isomorphism which
fixes the image of the injections in cohomology obtained from
including the complements into the full complex torus). Last, we give
another arrangement (also in rank $2$) which shows that a ``natural''
condition for the cohomology ring to be generated in degree $1$ is not sufficient.

\section{Preparations}\label{sec:preparation}
\subsection{Categories and diagrams}

Given a category $\Cc$, we will denote by $\gre{\Cc}$ the geometric realization
of the nerve of $\Cc$ (in particular, this is a polyhedral complex in
the sense of \cite{Kozlov}). A
functor $F: \Cc_1\to \Cc_2$ induces a cellular map $\gre{F}$ between
the geometric realizations. We will for brevity say that two categories are `homotopy
equivalent' meaning that their nerves are.

A kind of categories of special interest for us are {\em face
  categories of polyhedral complexes}.
We refer e.g.\ to \cite[Section 3]{dD2013} for a
precise definition and here only recall that the face category $\Fc(K)$ of a
polyhedral complex $K$ has the cells of $K$ as objects, and one
morphism $P\to Q$ for every attachment of the polyhedral cell $P$ to a face
of the polyhedral cell $Q$. 

It is a standard fact
that, if $K$ is a polyhedral complex, $\gre{\Fc(K)}$ can be embedded
into $K$ as its barycentric subdivision (see \cite{Tamaki} for a
thorough investigation of this situation).

Face categories of polyhedral complexes are examples of categories
where the identity morphisms are the only invertible morphisms, as
well as the only endomorphisms. Such categories are called {\em
  scwols} (for {``small
  categories without loops''}) in the terminology of \cite{BrHa} or
``acyclic categories'', e.g., in \cite{Kozlov}.
 
It is standard to consider a partially ordered set $(P,\leq)$ as
a scwol $\mathcal P$ with set of objects $\operatorname{Ob}(\mathcal P) =P$ and
$\vert\operatorname{Mor}_{\mathcal P}(p,q)\vert\leq 1$ for
all $p,q\in P$, with $\vert \operatorname{Mor}_{\mathcal P}(p,q) \vert
=1$ if and only if $p\leq q$: in this case we will simply speak of
``the morphism $p\leq q$''.

A {\em diagram} over a category $\Ic$ (which in our case will always
be a scwol) is a
functor 
$$
\mathscr D : \Ic \to \mathcal X
$$
where, in this paper, $\mathcal X$ can be the category $\tops$ of topological
spaces, $\scwols$ of scwols, or the categories of Abelian groups or
$\mathbb Z$-algebras. A morphism between diagrams $\mathscr D_1$, $\mathscr D_2$ over the same index category $\Ic$
is a family $\alpha=(\alpha_i : \mathscr D_1(i) \to \mathscr D_2(i))_{i\in \operatorname{Ob}\Ic}$ of morphisms of
$\mathcal X$ that commute with diagram maps - that is, such that, for every morphism $i\to j$ of $\Ic$,
$\alpha_j \circ \mathscr D_1(i\to j) = \mathscr D_2(i\to j) \circ \alpha_i$.

\subsubsection{$\mathcal X = \tops$} There is an extensive literature
on diagrams of spaces, in particular studying their homotopy
colimits. We content ourselves with listing some facts we'll have use
for and refer to \cite{WZZ} or \cite{Kozlov} for an introduction to the
subject and proofs.

\begin{lem}
  Let $\alpha$ be a morphism between two diagrams $\mathscr D_1,
  \mathscr D_2$ over the same index category $\Ic$. If every
  $\alpha_i$ is a homotopy equivalence, then $\alpha$ induces a
  homotopy equivalence of homotopy colimits
$$
\hocolim \mathscr D_1 \to \hocolim \mathscr D_2.
$$
\end{lem}

That there is a
canonical projection 
$$
\pi: \hocolim \mathscr D \to \gre{\Ic}
$$

The Leray spectral sequence of this projection then can be used to
compute the (co)homology of the homotopy colimit. It is equivalent to
the spectral sequence studied by Segal \cite{segal68} and has second
page
$$
E_2^{p,q} = H^p(\gre{\Ic}, \mathscr H^q(\pi^{-1}; \mathbb Z))
\Rightarrow H^{*} (\hocolim \mathscr D ; \mathbb Z).
$$

\subsubsection{$\mathcal X = \scwols$}
\def\Ob{\operatorname{Ob}}
\def\Mor{\operatorname{Mor}}

The topological spaces we will be studying will come with a natural
combinatorial stratification and can therefore can be written as nerves
of acyclic categories.
Recall from \cite[Definition 1.1]{Thomason} the {\em Grothendieck
  construction} $\GC \mathscr D$ associated to a diagram $\mathscr D :
\Ic \to \scwols$. This is the category with object set consisting of all pairs $(i,x)$ with $i\in
\Ob(\Ic)$ and $x\in \Ob(\mathscr D(i))$, and with morphisms
$(i_1,x_2)\to (i_2,x_2)$ corresponding to pairs $(f,\mu)$,
$f\in\Mor_{\Ic}(i_1,i_2)$ and $\mu \in \Mor_{\mathscr
  D(i_2)}(f(i_1),i_2)$, composed in the obvious fashion.
  
\begin{lem}[Theorem 1.2 of \cite{Thomason}] \label{lem:Thomason}
Given a diagram $\mathscr D : \Ic \to \scwols$, we have a natural
homotopy equivalence
$$
\hocolim \gre{\mathscr D} \simeq \gre{\GC{\mathscr D}}.
$$
\end{lem}

\begin{rem}\label{rem:projGC}
  In this case the canonical projection of the homotopy colimit on the
  nerve of the index category becomes the map of polyhedral complexes
  induced by the evident functor
$$\GC \mathscr D \to \Ic;\quad\quad (i,x) \mapsto i$$
\end{rem}

\subsection{Arrangements of hyperplanes}
Let $\A$ be a locally finite arrangement of hyperplanes
in $\C^d$. We will write
$$
M(\A) := \C^d\setminus \bigcup \A
$$
for its complement.

Recall the definitions of Section \ref{sec:backarr} and, given $X\in
\Ll(\A)$, the arrangement $\A_X:=\{H\in \A \mid X\subseteq H\}$
(Equation \eqref{eq:1}). If  $\A$ is complexified, the associated real
arrangement $\A_{\mathbb R}$ induces a polyhedral cellularization of $\R^d$
with poset of {\em faces} $\Fc(\A)$, ordered by inclusion, whose maximal elements (the maximal
cells) are called {\em chambers} of $\A$. We write $\Tc(\A)$ for the
set of all chambers of $\A$. 

Notice that every $G\in \Fc(\A)$ is contained in a unique (relatively open) face
of $\A_X$, that we denote by $G_X$. One readily checks that
this defines
a poset map $ \Fc(\A) \to \Fc(\A_X)$, since $F_1\geq F_2$ implies $(F_1)_X \geq (F_2)_X$.

\subsubsection{Sign vectors and operations on faces} A standard way of dealing with such
polyhedral subdivisions is by choosing a real defining form $\ell_H$ for
every $H\in \A$ and thus defining $H^+:=\{x\in \mathbb R^d \mid
\ell_H(x)>0\}$, $H^-:=\{x\in \mathbb R^d \mid \ell_H(x)<0\}$,
$H^0:=H$.
Each face $F$ is then identified by its {\em sign vector} $\sv{F}:\A\to
\{+1,-1,0\}$ with $\sv{F}(H):=\sigma$ if and only if $F\subseteq H^{\sigma}$.
If we consider the set $\{+1,-1,0\}$ partially ordered according to
$+1 >0$, $-1>0$ and $+,1,-1$ incomparable we see that, with our notation, for any $F,G\in \Fc(\A)$ we
have 
$$
F\leq G \textrm{ if and only if, for all }H\in \A,\,\, \sv{F}(H) \leq \sv{F}(H).
$$
Also, for every $X\in \Ll$ we have that $\sv{F_X}$ is the restriction
of $\sv{F}$ to $\A_X$. Given chambers $C_1,C_2 \in \Tc(\A)$, recall the set 
$$
S(C_1,C_2)=\{H\in \A \mid \sv{C_1}(H) = - \sv{C_2}(H) \neq 0\}
$$
of hyperplanes {\em separating} $C_1$ from $C_2$.

For $F\in \Fc(\A)$ we let $\A_F:=\A_{\vert F\vert}$, where
$\vert F \vert$ denotes the affine span of $F$. This will then mean
that $H\in \A_F$ if and only if $\sv{F}(H)=0$.

\begin{df}\label{rem:identiF}
Given $F,G\in\Fc(\A)$, we define $G_F\in \Fc(\A)$ to be the face
uniquely determined by $(G_F)_{\vert F \vert}=G_{\vert F \vert}$ and $G_F \geq F$. 

In particular, there is an inclusion
$$
i_F: \Fc(\A_{F})\to \Fc(\A) 
$$
where in terms of sign vectors we have
$$
\sv{G_F}(H):=\left\{
  \begin{array}{ll}
    \sv{F}(H) & \textrm{ if } H \not\in \A_F \\
    \sv{G}(H) & \textrm{ if } H \in \A_F \\
  \end{array}
\right.
$$
and
$$
\sv{i_F(G)}(H):=\left\{
  \begin{array}{ll}
    \sv{F}(H) & \textrm{ if } H \not\in \A_F \\
    \sv{G}(H) & \textrm{ if } H \in \A_F \\
  \end{array}
\right.
$$
\end{df}

The following are some properties that show that the above objects are
well-defined, and which we list as a lemma for later reference. Their proof is a straightforward check of sign vectors.
\begin{lem}\ \label{piccolezze}
\begin{itemize}
\item[(1)] $(i_F (G))_{\vert F \vert}= G$, hence $i_F$ maps
  bijectively onto $\Fc(\A)_{\geq F}$.
\item[(2)] If $G_1\geq G_2 \in \Fc(\A_F)$, then $i_F(G_1) \geq
  i_F(G_2)$.
\item[(3)] $(G_{F_1})_{F_2} = G_{{F_1}_{F_2}}$ for all $G,F_1,F_2\in
  \Fc(\A)$.
\item[(4)] $F_G=F$ if $F\geq G$.
\end{itemize}
\end{lem}

\begin{df}
  Let $\A$ be a complexified real central arrangement, $\mathcal X\subset \A$ and
  $\sigma\in \{\pm 1,0\}$. Define
$$
\Delta^{\sigma} (\A;\mathcal X):=\{F\in \Fc(\A) \mid \sv{F}(H) = \sigma \textrm{ if }H\in
X\}
$$ 

For disjoint
  (possibly empty) subsets $N,Z,P\subseteq \A$ define then
$$
\Delta(\A;N,Z,P):= \Delta^{-1}(\A;N)\cap \Delta^{0}(\A;Z)\cap \Delta^{+1}(\A;P).
$$
\end{df}

We will have occasional use for the following result.

\begin{lem}[See Proposition 4.3.6 of \cite{BLSWZ}]\label{lem:conv}
  Let $\A$ be a nonempty, central complexified arrangement. The subposet $\Delta(\A;N,Z,P)$, if not empty,  is contractible.
\end{lem}

\begin{proof}
  If $N=P=\emptyset$, the posets under consideration contain the unique
  minimal element $\cap \A$ thus are contractible. Otherwise, our
  $\Delta(\A;N,Z,P)$ corresponds to the poset of cells in the interior
  of the shellable PL-ball denoted by  $\Delta^-_{N\cup Z} \cap
  \Delta^+_{P\cup Z}$ in
  \cite[Proposition 4.3.6.]{BLSWZ}, and we conclude with the general
  fact that the order complex of the poset $\mathcal P$ of cells of the interior of
  a PL-ball $B$ is contractible. This last fact can be proved as follows:
  the order complex of $\mathcal P$ is the maximal subcomplex of all
  cells of $\Delta
  (\Fc(B))$ that are disjoint from the (full) subcomplex $\Delta(\Fc(B)
  \setminus \mathcal P)$ (which triangulates the boundary of $B$). Hence
  $\Delta(\mathcal P)$ is a retract of the interior of $B$ and, as such, contractible.
\end{proof}

\subsection{A combinatorial Brieskorn lemma}
  The data of $\Fc(\A)$ can be used to construct a regular CW-complex
  due to Salvetti \cite{Salvetti} which embeds in $M(\A)$ as a
  deformation retract. This complex is called {\em Salvetti complex}
  of $\A$ and denoted $\Sal(\A)$. Its face category (in fact, a poset)
  $\Sc(\A):=\Fc(\Sal(\A))$ can be described as follows:

\begin{align*}
  \Sc(\A) = \{[F,C]\in \Fc(\A) \times \Tc(\A) \mid F\leq C \textrm{ in
  }\Fc(\A) \} \\
  [F,C] \geq [F',C'] \textrm{ if } F\leq F', \, C_{F'}=C'.
\end{align*}

\begin{df}\label{def:SC}
  From now in this section we will assume that the arrangement $\A$ is
  {\em central}, i.e., that $\cap \A \neq
  \emptyset$. Then, letting $P:=\cap \A$, the complex $\Sal(\A)$ can be decomposed as a union of
  (combinatorially isomorphic) closed polyhedral cells of dimension
  $d$, corresponding to the pairs $[P,C]$ with $C\in \Tc(\A)$. We
  define subposets
$$
\Sc_C:= \Sc_{\leq [P,C]}
\quad\textrm{for} \quad
C\in \Tc(\A)
$$
corresponding to the faces of the closure of the maximal cells $[P,C]$.
\end{df}

Our next goal will be to offer a combinatorial version of 
Lemma \ref{lem:brieskorn}, i.e., to express Brieskorn's map as induced
by poset maps between
Salvetti complexes. 

 \begin{df}\label{df:jb}
   Given $X\in \Ll(\A)$ define
   $$
   b_X:\Sc(\A) \to \Sc(\A_X), [G,C] \mapsto [G_X,C_X]
   $$
Moreover, for every $F\in \Fc(\A)$  we have the following natural inclusion
 of posets, well-defined
 by Lemma \ref{piccolezze}.(2).
$$j_F:
 \Sc(\A_F) \to \Sc(\A), \quad [G,C] \mapsto [i_F(G), i_F(C)].
$$
 \end{df}

The following is then a combinatorial version of Brieskorn's Lemma.

\begin{prop}[Combinatorial Brieskorn Lemma]\label{cor:cbl}
Let $\A$ be a central complexified arrangement of hyperplanes and for every $X\in\Ll(\A)$ choose an  $F(X)\in \Fc(\A)$ with
$\vert F(X)\vert = X$.

 The maps of posets $ j_{F(X)}, b_X$ induce an
  injective map $b_X^* : H^*(\Sal(\A_X)) \to H^*(\Sal(\A))$ and a surjective map 
  $j_{F(X)}^*: H^*(\Sal(\A)) \to H^*(\Sal(\A_F))$, such that $j_F^* \circ
  \hat b_X = \id_{H^*(\Sal(\A_X))}$.

Moreover,  Then for all $k$ the inclusion
$$
j_k (= \sqcup_{\rk(X)=k} j_{F(X)}): \bigsqcup_{\rk(X)=k} \Sc(\A_{F(X)}) \to \Sc(\A)
$$ 
induces the Brieskorn isomorphism
$$
(b^k)^{-1}: H^k(\Sc(\A)) \to \bigoplus_{\rk(X)=k} H^k(\Sc(\A_{F(X)})) = \bigoplus_{\rk(X)=k} H^k(\Sc(\A_{X})). 
$$
In particular, the map induced in cohomology by $j_{F(X)}$ does not depend
on the choice of $F(X)$ among the maximal cells of its affine span.
\end{prop}

\begin{proof}
  For the first part of the claim, notice (e.g., by a check of sign vectors) that the composition
 $b_X\circ j_{F(X)}$ is the identity on $\Sc(\A_{F(X)})$. 
 For the second part we prove that, in fact, the map $b_X$ is homotopic to the
  inclusion $M(\A) \subseteq M(\A_X)$.

First of all, notice that the radial map $\rho: z\mapsto z/\vert z \vert$
defines a homotopy between the inclusion $M(\A)\subseteq M(\A_X)$ and
the inclusion $S \setminus \A \subseteq S\setminus \A_X$, where $S$
denotes the unit sphere in $\mathbb C^d \simeq \mathbb R^{2d}$.

We follow \cite[Chapter 5]{OT} and consider the arrangements $\A$ and $\A_X$ as framed by the
arrangement $$\HH=\A_1 \cup \A_2 := \{H \times \mathbb R^d \mid H
\in \A_{\R}\}\cup \{ \mathbb R^d\times H \mid H
\in \A_{\R}\}$$ in $\R^{2d}$. This defines, as usual, a
  cellularization of $\R^{2d}$ with poset of faces $\Fc(\HH) \simeq \Fc(\A)
  \times \Fc(\A)$ (product of posets, see
  e.g. \cite[Section 3.2]{Stanley}) 
and, after barycentric subdivision, a
  triangulation $T_\HH$ of the sphere $S$ realizing the order complex
  of $\Fc(\HH) \setminus \{ \hat 0\}$. The intersection $S\cap \bigcup \A$ (resp. $S\cap
  \bigcup \A_X$) is a full
  subcomplex $N_\A$ (resp. $N_{\A_X}$) of $T_\HH$, thus $N_{\A_X}
  \subseteq N_{\A}$ as a full subcomplex.
  Let $M_\HH(\A)$ be the biggest subcomplex of $T_\HH$ which is
  disjoint to $N_\A$, and similarly for $M_\HH(\A_X)$. Then,
  $M_\HH(\A) \subseteq M_\HH(\A_X)$ is a full subcomplex.

  It is a standard fact (see e.g. \cite[Lemma 70.1]{Munkres}) that
  $T_\HH \setminus \bigcup \A_X$ deformation retracts onto $M_\HH(\A_X)$
  (say, by a retraction $f_X$) and $M_\HH(\A_X) \setminus \A$ deformation retracts
  onto $M_\HH (\A)$ (say, by $f$). We then have that the inclusion
  $M_\HH(\A) \subseteq M_\HH(\A_X)$ is homotopic to the original
  inclusion $M(\A) \subseteq M(\A_X)$.

  Now notice that the simplicial complexes $M_\HH(\A_X)$, $M_\HH(\A)$
  are in fact realizations of the order complexes of the posets 
$$
\MM_\HH(\A_X) =\{ (F,G)\in \Fc(\A) \times \Fc(\A) \mid \A_F \cap \A_G
\cap \A_X=\emptyset\}
$$
$$
\MM_\HH(\A) =\{(F,G)\in \Fc(\A) \times \Fc(\A) \mid \A_F \cap \A_G =\emptyset\}
$$
and the inclusion of complexes is induced by the inclusion of posets
$\iota: \MM_\HH(\A) \to \MM_\HH(\A_X)$. We summarize by saying that
the following diagram commutes up to homotopy

$$\begin{CD}
  M(\A) @>{\rho_{\vert M(\A)}}>> T_\HH \setminus N_\HH(\A) @>{f_X}>>
  M_\HH(\A_X)\setminus N_\HH(\A)  @>{f}>> M_\HH (\A)\\ 
  @V{\subseteq}VV @V{\subseteq}VV @V{\subseteq}VV @V{\gre{\iota}}VV\\
  M(\A_X) @>{\rho}>> T_\HH \setminus N_\HH(\A_X) @>{f_X}>>
  M_\HH(\A_X)  @= M_\HH (\A_X) \\ 
\end{CD}
$$

In order to study the map $\gre{\iota}$ further, it is enough to argue
at the level of posets.

In \cite[Chapter 5]{OT} is proved that the map
$$\phi: \MM_\HH(\A) \to \Sc(\A)^{op}, \quad (F,G) \mapsto [F,G_F] $$
is a homotopy equivalence.

We define a map
$$\psi: \MM_\HH(\A_X) \to \Sc(\A_X)^{op}, \quad (F,G) \mapsto [F_X,(G_F)_X] $$
so that, by definition, the following diagram commutes:

$$
\begin{CD}
  \MM_\HH(\A) @>{\phi}>> \Sc(\A) ^{op} \\
  @V{\iota}VV @V{b_X}VV\\
  \MM_\HH(\A_X) @>{\psi}>> \Sc(\A_X)
\end{CD}
$$

Now it
is enough to prove that $\psi$ is a homotopy equivalence, and we will
then have proved that the geometric map $\gre{b_X}$ induced by $b_X$ is homotopic to $\gre{\iota}$, which
in turn is homotopic
to the inclusion $M(\A)\subseteq M(\A_X)$.

To prove that $\psi$ is a homotopy equivalence, consider some
$[F,C]\in \Sc(\A_X)^{op}$ and $(F',G)\in   \MM_\HH(\A_X) $ such that $\psi(F',G) \geq
[F,C]$ in $\Sc(\A_X)^{op}$. This is the case exactly if $F'_X\geq F$ in $\Fc(\A)$ and
$(G_{F'})_X=C_{(F')_X}$. 

In terms of sign vectors, this will be
verified exactly if:
\begin{itemize}
\item $\sv{F'}(H)\geq \sv{F}(H)$ for all $H\in \A_X$ 
\item $\sv{G}(H) = \sv{C}(H)$ for all $H\in\A_X\cap\A_{F'}$
\item $\A_{F'}\cap \A_X \cap \A_G =\emptyset$
\end{itemize}
where the last condition just ensures that indeed $(F',G)\in \MM_\HH(\A_X)$.

Going back, we see that, under the isomorphism $(\Fc(\A))^2 \simeq
\Fc(\HH)$ these $(F',G)$ are exactly the faces $\widehat F$ of $\HH$
with 
$$
\sv{\widehat F}(H\times \R^d) = \gamma_F(H) 
\textrm{ for } H\in \A_X \setminus \A_F, \quad
\sv{\widehat F}(\R^d \times H) = \gamma_C(H) 
\textrm{ for } H\in \A_X,
$$
and thus $\psi^{-1}(\Sc(\A_X)_{\geq}[F,C])$ is a subposet of
$\Fc(\HH)$ consisting of all faces with prescribed sign on a certain
set of hyperplanes. This set is of the form $\Delta (\HH;N,Z,P)$, nontrivial because $X\neq \hat 0$ and
nonempty because it contains $(F,C)$, and
their order complex is thus contractible by Lemma \ref{lem:conv}.
\end{proof}

We next prove a proposition which expresses, in the language of posets, the
fact that given any face $F$ of a central arrangement, the union of the cells $[P,C]$
with $C$ running through all chambers adjacent to $F$ is a subcomplex
of $\Sal(\A)$ homotopy equivalent to $M(\A_F)$.

\begin{df}\label{def:j}   Let $\A$ be a central, complexified arrangement of hyperplanes, and
  write $P$ for the minimal element of $\Fc(\A)$. We define a subposet
  of $\Sc(\A)$ as
$$\SF{F}(\A) := \bigcup_{C\geq F} \Sc_{C}$$
(where we view $\Sc_C$ as a subposet of $\Sc(\A)$ as in Definition \ref{def:SC}.), 
consider the restriction 
$$
j_0^F: \Sc(\A_F) \to \SF{F}(\A)
$$
of the map $j_F$ of Definition \ref{df:jb} and define 
$$
 \xi_F: \SF{F}(\A) \to  \Sc(\A_{\vert F\vert}), \quad [G,C] \mapsto [G_{\vert F\vert},C_{\vert F\vert}].
$$
\end{df}

\begin{prop}\label{xi:homeq}
  The poset maps $j_0^F$, $\xi_F$ are homotopy inverse to each other.
\end{prop}

The proof of this proposition will rest on some technical facts about
the maps $\xi_F$ that
we prove as separate lemmas for later reference.

\begin{lem}\label{lem:preim}
  Let $F\in \Fc(\A)$, $C,G\in\Fc(\A_{\vert F \vert})$ with $C$ a chamber. Then
  $$
  \xi_F^{-1}([G,C]) =\{[K,\iF{F}{C}_K] \in \Sc_C \mid K_{\vert F\vert}=G\}. 
  $$
\end{lem}
\begin{proof}
  We first prove the right-to-left inclusion. For $K$ such that
  $K_{\vert F\vert }=G$ we have $$\xi_F([K,\iF{F}{C}_K])=[K_{\vert F\vert },(\iF{F}{C}_{K})_{\vert F\vert }]=[K_{\vert F\vert },\iF{F}{C}_{K_{\vert F\vert }}]=[K_{\vert F\vert },C_G]=[G,C]$$
  where in the second equality we used Lemma \ref{piccolezze}.(3) and in the last
  equality simply the fact that by definition $C\geq G$.

  Now to the left-to-right inclusion. Consider $K\in\Fc(\A)$ and
  $R\in\Tc(\A)$ with $K\leq R$ and such that
  $\xi_F([K,R])=[G,C]$. Immediately by definition we have $K_F=G$,
  and we are left with proving that $R=\iF{F}{C}_K$. For that, we check the
  definitions.
  \begin{itemize}
  \item[-] If  $\sv{K}(H)\neq 0$, $\sv{\iF{F}{C}_K}(H) = \sv{K}(H)$
    and, since $R\geq K$, $\sv{K}(H)=\sv{R}(H)$.
  \item[-] If $\sv{K}(H) =0$,
 $$
\sv{\iF{F}{C}_K}(H) = \left\{\begin{array}[l]{ll}
                 \!\! \sv{F}(H) & \!\textrm{if } \sv{F}(H) \neq 0,\\
        \!\!\sv{\iF{F}{C}}(H)\!=\!\sv{R}(H) & \!\textrm{else, because
        } C\!=\!R_{\vert F\vert}.
      \end{array}\right.
 $$

It remains to see that $\sv{F}(H)=\sv{R}(H)$ when $\sv{K}(H)=0$ and
$\sv{F}(H)\neq 0$. Indeed, since $[K,R]\in \Sc^F(\A)$, it must be
$R=(C')_K$ (hence $\sv{R}(H)=\sv{C'}(H)$ when $\sv{K}(H)=0$) for some
$C'\geq F$ (thus $\sv{C'}(H) = \sv{F}(H)$ whenever $\sv{F}(H)\neq
0$). 
  \end{itemize}
\end{proof}

\begin{cor}\label{cor:pre-cont}
  For every $S\in \Sc(\A_{\vert F\vert})$, the poset $\xi_F^{-1}(S)$ is contractible.
\end{cor}

\begin{proof}
  The expression given in Lemma \ref{lem:preim} shows that
  $\xi_F^{-1}([G,C])$ is poset-isomorphic to the order dual of the subposet of $\Fc(\A)$
  consisting of all $K\in \Fc(\A)$ with $K_{\vert F \vert} = G$:
  indeed, given two such $K_1,K_2$ with $K_1\geq K_2$, then 
 $(\iF{F}{C}_{K_1})_{K_2}=\iF{F}{C}_{{K_1}_{K_2}} =\iF{F}{C}_{K_1}$,
 hence $[K_1,\iF{F}{C}_{K_1}] \leq [K_2,\iF{F}{C}_{K_2}]$, and the
 reverse implication is trivial.

  Now, the $K\in \Fc(\A)$ with $K_{\vert F \vert}=G$ are exactly those
  in the subposet $$\Delta(\A, \A_F\cap \gamma_G^{-1}(-1),
  \A_F\cap \gamma_G^{-1}(0),\A_F\cap \gamma_G^{-1}(+1)),$$ which is
  nonempty  thus contractible by Lemma \ref{lem:conv}.
\end{proof}

\begin{lem}\label{lem:fl}
  Let $\A$ be a central, complexified arrangement of hyperplanes. 
 For every $F\in \Fc(\A)$ and every $[G,C]\in \Sc(\A_F)$, the poset
 $\xi_F^{-1}(\Sc(\A_F)_{\leq [G,C]})$ is contractible.
\end{lem}

\begin{proof} 
Consider an element $[G,C]\in \Sc(\A_F)$
(thus $G\geq F$ and $C\geq G$ in $\Fc(\A)$) and consider the
preimage of 
$$\Sc(\A_F)_{\leq [G,C]}=\{[G',C'] \vert G'\geq G, C_{G'}=C'\}$$ 
with respect to $\xi_F$.

By Lemma \ref{lem:preim}, this preimage is the subposet of $\Sc^F$
consisting of elements
$$
\bigcup_{
\substack{G'\geq G \\ C' = C_{G'}}
}\{[K,C'_K] \mid K_{\vert F\vert} = G'\} = \{[K,R_K] \mid K_{\vert
  F\vert}\geq G, R=i_F(C_{G'})=i_F(C_{K_{\vert F\vert}})\}
$$
which is isomorphic, as in the proof of Corollary \ref{cor:pre-cont} to the subposet of $\Fc(\A)$ given by 
$$
\Pc :=\{K \in \Fc(\A) \mid K_{\vert F\vert}\geq G\} = \Delta(\A;
\A_F\cap G^-;\emptyset; \A_F \cap G^+),
$$
which is nonempty (it contains e.g.\ $\iF{F}{G}$), hence contractible by Lemma \ref{lem:conv}.
\end{proof}

\begin{proof}[Proof of \ref{xi:homeq}]
  The composition $\xi_F \circ j_0^F$ equals obviously the identity. We
  prove that $j_0^F \circ \xi_F$ is homotopic to the identity on
  $\SF{F}$.
To this end consider
$$
\alpha: \Delta(\SF{F}) \to 2^{\gre{\SF{F}}}, \quad \alpha(\sigma):= \gre{\xi^{-1}_F(\Sc(\A_F)_{\leq \max\xi_F(\sigma)})}.
$$
Clearly, the carrier map $\alpha$ carries the identity. Moreover,
an easy check shows that
$$
(\xi_F \circ j_0^F\circ \xi_F)(\sigma) =\xi_F(\sigma)
$$
and thus 
$$
( j_0^F\circ \xi_F)(\sigma) \subseteq \xi^{-1}(\Sc(\A_F)_{\leq \max \xi_F(\sigma)}),
$$
hence $\alpha$ carries both the identity and $j_0^F\circ\xi_F$. We
conclude by
the Carrier Lemma \cite[Proposition II.9.2]{LW} and Lemma \ref{lem:fl}.
\end{proof}

\section{Combinatorial topology of toric arrangements} \label{sec:tsc}

\subsection{The toric Salvetti category}
Let $\A$ be a complexified
toric arrangement. 
One way to obtain an analogue of
Salvetti's complex is to notice
that the canonical embedding of $\Sal(\Aup)$ into $M(\Aup)$ is equivariant with
respect to the action of the rank-$d$ integer lattice on $\C^d$ as the group
of deck transformations of the universal cover of $T$. This leads us to
look for a convenient description of the quotient of the Salvetti
complex, as was first done in \cite{MociSettepanella} in the case
where the resulting complex is again simplicial. In general, one sees that this action restricts to an
action on $\Fc(\Aup)$, and the face category $\Fc(\A)$ is isomorphic to the quotient
$$
\Fc(\A) \cong \Fc(\Aup)/\mathbb Z^d, \textrm{ with covering functor
} Q: \Fc(\Aup) \to \Fc(\A).
$$
This action induces an action on $\Sc(\Aup)$ via

\begin{equation}
g[F,C] := [gF,gC]\label{eq:action}
\end{equation}

for every $F,C\in \Fc(\Aup)$ with $C\geq F$, and hence to a cellular
action on $\Sal(\Aup)$. In \cite{dD2010}, taking advantage of
the generality of acyclic categories, a
description of the quotient category $\Sc(\Aup)/\Z^d$ is given,
together with the proof that indeed $\gre{\Sc(\Aup)/\Z^d}\simeq
\gre{\Sc(\Aup)} / \mathbb Z^d \simeq M(\A)$. 

\bigskip

In fact, much of the notation of the previous section has been
introduced in \cite{dD2010,dD2013} in order to describe the
relationships involved in these covering morphisms. For instance, if
$F\in \Fc(\A)$, the arrangement $\A[F]$ is an `abstract' copy of every
$\Aup_{F^\upharpoonright}$ with $Q(F^\upharpoonright) = F$ for which
we can choose linear forms according to those defining $\Aup$. Then, there
are canonical isomorphisms
\begin{equation}
\Fc(\Aup)_{\geq F^\upharpoonright} \cong \Fc(\Aup_{F^\upharpoonright}) = \Fc(\Ar{F}), \quad\quad \Sal(\Aup_{F^\upharpoonright})\cong\Sal(\Ar{F}).\label{eq:iso}
\end{equation}
given by the identical mapping of sign vectors.
If $m:F\to G$ is a morphism of $\Fc(\A)$, then 
 $\A[G]= \A[F]_{G_0}$ where $G_0$ is the intersection of all hyperplanes of $\A[F]$ that correspond to hypertori
  containing $G$. Moreover, for any choice of $F^\upharpoonright \in
  Q^{-1}(F)$ 
there is exactly one face $G^\upharpoonright$ such that
  $Q(F^\upharpoonright \leq G^\upharpoonright)=m$: we call $F_m$ the
 corresponding   face of $\A[F]$  under the isomorphism of Equation
 \eqref{eq:iso} (notice that $\vert F_m\vert = G_0$). Following \cite{dD2010} we define
$$
i_m:  
\Fc(\A[G]) \hookrightarrow \Fc(\A[F])
$$
to be the map corresponding to the inclusion $\Fc(\Aup)_{\geq
  G^\upharpoonright}\subseteq \Fc(\Aup)_{\geq F^\upharpoonright}$.
\begin{rem}\label{rem:im}
 In terms of sign vectors, the map $i_m$ is determined as follows:
$$
\sv{i_m(K)}(H):=\left\{
  \begin{array}{ll}
    \sv{F_m}(H) & \textrm{ if } H\not\in \Ar{G} \\
    \sv{K}(H) & \textrm{ else.}
  \end{array}
\right.
$$
  In particular, if $X$ is a flat of both $\A[G]$ and of $\A[F]$, then $i_m(K)_X =
  K_X$ for all $K\in \Fc(\A[G])$.
\end{rem}

The order relation $F^\upharpoonright\leq G^\upharpoonright$ also
defines an inclusion
$$
j[F^\upharpoonright\leq G^\upharpoonright]:
\Sal(\Aup_{G^\upharpoonright}) \hookrightarrow \Sal(\Aup_{F^\upharpoonright}),
$$
i.e., the map induced on complexes by the   inclusion $j_G$ of Definition \ref{df:jb}
with respect to the ``ambient'' arrangement
$\A^{\upharpoonright}_{F}$. 

Given a morphism $m:F\to G$ of $\Fc(\A)$, there is a
corresponding inclusion induced by $i_m$ (see also \cite[Definition 5.9]{dD2013}),
$$
 j_m: 
\Sc(\A[G]) \hookrightarrow \Sc(\A[F]), [K,C]\mapsto [i_m(K),i_m(G)].
$$
\bigskip
\begin{figure}[t]\label{fig:notazione}
  \input{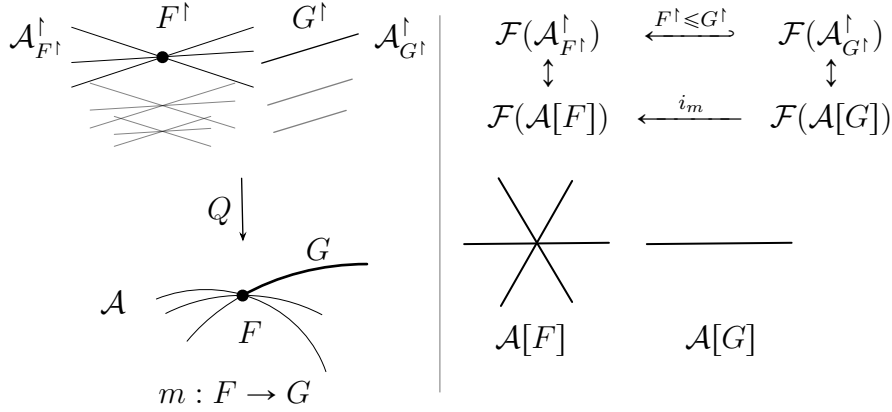}
\caption{A pictorial representation of the notations introduced here.
The ``abstract'' arrangements $\Ar{F}$ serve as a model for the
localizations $\Aup_{F^\upharpoonright}$, and the map $i_m$ is induced
at the abstract level by the inclusion defined by a(ny) lift of the morphism $m$ to $\Fc(\Aup)$.}
\end{figure}
We can now re-state the following definition from \cite{dD2013}.

\begin{df}
  Let $\A$ be a complexified toric arrangement, recall the
  notations of Definition \ref{df:A0} and, for $F\in \Fc(\A)$, write
  $\Ar{F}:=\Ar{\vert F \vert}$, where
  $\vert F\vert$ denotes the smallest dimensional layer that contains $F$. Define a diagram
  \begin{equation*}
    \mathscr D : 
    \begin{array}[t]{rcl}
     \Fc(\A)^{op} & \to & \scwols \\
    F& \mapsto & \Sc(\Ar{F})\\
     m: F\to G & \mapsto & j_m: \Sc(\Ar{G}) \hookrightarrow \Sc(\Ar{F})
   \end{array}
  \end{equation*}
\end{df}

In \cite{dD2013} it is proved that $\gre{\colim \mathscr D}$ is
homotopy equivalent to $M(\A)$. For our purposes we need to prove
another connection between $\mathscr D$ and the homotopy type of
$M(\A)$.

\begin{thm}
  $\hocolim \gre{\mathscr D} \simeq M(\A) $.  
\end{thm}

\begin{proof}
\newcommand{\ZZZ}{\mathbb Z^d}
  We will prove that $\hocolim \gre{\mathscr D}$ is homotopy
  equivalent to the quotient of
  $\Sal(\A^{\upharpoonright})$ by the induced $\ZZZ$-action.

Write $\ZZZ$ for the one-object category (we write $\ast$ for
  this object) representing the group. We then write the quotient $\Sal(\A^{\upharpoonright})$ as
  the colimit of the diagram
  \begin{align*}
    \mathscr S : \ZZZ &\to \tops \\
    \ast &\mapsto \Sal(\A^\upharpoonright) \\
    g &\mapsto g: \Sal(\A^\upharpoonright)\to \Sal(\A^\upharpoonright)
  \end{align*}
  here the action of some $g\in \ZZZ$ on $\Sal(\A^{\upharpoonright})$
  is described in Equation \eqref{eq:action}.

Now  notice that by construction 
  $\Sal(\A^{\upharpoonright})$ is covered by the subcomplexes
  $\Sal(\A^{\upharpoonright}_{F})$ and is thus the colimit of 
  \begin{align*}
    \mathscr E : \Fc(\A^{\upharpoonright})^{op} \to \tops,\quad
     F \mapsto \Sal(\A^{\upharpoonright}_{F}),\quad
     F\leq G \mapsto j[F\leq G]: \Sal(\A^{\upharpoonright}_{G}) \to \Sal(\A^{\upharpoonright}_{F}),
  \end{align*}
where $j[F\leq G]$ denotes the map induced on complexes by the   inclusion $j_G$ of Definition \ref{df:jb}
with respect to the ``ambient'' arrangement
$\A^{\upharpoonright}_{F}$. Consider now the push-forward $\hat{\mathscr E}$ of $\mathscr
E$ along  the functor $Q: \Fc(\A^\upharpoonright) \to \Fc(\A),$ 
(and thus with $\colim \mathscr E = \colim \hat{\mathscr E}$). We have
the following explicit form.

    \begin{align*}
    \hat{\mathscr E} : \Fc(\A)^{op}& \to \tops \\
    F &\mapsto \bigsqcup_{F^\upharpoonright \in Q^{-1}(F)}\Sal(\A^{\upharpoonright}_{F^{\upharpoonright}}) \\
     m: F\to G &\mapsto \sqcup_{(F^\upharpoonright\leq
       G^\upharpoonright) \in Q^{-1}(m)} j[F^\upharpoonright\leq
     G^\upharpoonright] : (S, G^\upharpoonright) \mapsto (j[F^\upharpoonright\leq
     G^\upharpoonright](S), F^\upharpoonright) 
  \end{align*}

Notice that, for $g\in \ZZZ$ and any morphism $F^\upharpoonright \leq
G^\upharpoonright$ of the (poset-) category $\Fc(\Aup)$, we have $g \circ j[F^\upharpoonright \leq
G^\upharpoonright] = j[g(F^\upharpoonright \leq G^\upharpoonright)]$,
where we write $g$ for the automorphism of $\Fc(\Aup)$ defined by $g$.
With this, the following diagram describing the action of $\ZZZ$ on
$\hat{\mathscr E}$ is well-defined.
  \begin{align*}
    \hat{\mathscr G} : \ZZZ \times \Fc(\A)^{op}& \to \tops \\
    (\ast, F) &\mapsto \hat{\mathscr E} (F) \\
     (g, m: F\to G) &\mapsto  
     g\circ \hat{\mathscr E} (m)
  \end{align*}

\begin{list}{}{}\item[{\bf Claim 1:}] $\colim_{\Fc(\A)} \hat{\mathscr G}
  = \mathscr S$.
\item[{\em Proof:}] We check the equality pointwise. On object(s) we
  have
$$
(\colim_{\Fc(\A)} \hat{\mathscr G}) (\ast) = \colim
\hat{\mathscr E} = \colim {\mathscr E} = \Sal(\Aup).
$$
The morphism $(\colim_{\Fc(\A)} \hat{\mathscr G}) (g):
(\colim_{\Fc(\A)} \hat{\mathscr G}) (\ast)\to (\colim_{\Fc(\A)} \hat{\mathscr G}) (\ast)$ is induced by
the natural transformation $\hat{\mathscr G}(g,\id_{-}): \hat{\mathscr G} (\ast,-)
\Rightarrow \hat{\mathscr G}(\ast,-)$ which acts over objects as

\begin{equation}
\hat{\mathscr G} (g,\id_{F^\upharpoonright}):
\begin{array}[t]{rll}
\hat{\mathscr G}(\ast,F) &\to&
\hat{\mathscr G}(\ast,F), \\ (S, F^\upharpoonright) &\mapsto & (g\circ
\id_F)(S, F^\upharpoonright)=(g(S),g(F^\upharpoonright))
\end{array}\label{eq:explicitG}
\end{equation}

Now one checks explicitly that this induces the map $S \mapsto
g(S)$ on $\colim \hat{\mathscr E} = \Sal(\Aup)$, as required. \hfill$\triangle$
\end{list}
\begin{list}{}{}\item[{\bf Claim 2:}] $\colim_{\ZZZ} \hat{\mathscr G}
  \cong \gre{\mathscr D}$.
\item[{\em Proof:}]
 Again we can verify the isomorphism pointwise using the fact that
 preimages under the functor $Q$ are exactly orbits of the action of
 $\ZZZ$ on $\Fc(\Aup)$. For every $F\in
 \operatorname{Ob}(\Fc(\A))$, with Equation \eqref{eq:explicitG} we have
$$
\colim_{\ZZZ} \hat{\mathscr G}(-,F) = \hat{\mathscr E}(F)/\ZZZ \cong
\Sal(\Ar{F}) \cong \gre{ \mathscr D (F)} ,
$$
where the last two congruence symbols denote isomorphism of complexes
and homeomorphism (by barycentric subdivision),
and on morphisms $m\in\operatorname{Mor}(\Fc(\A))$:
$$
\colim_{\ZZZ} \hat{\mathscr G}(-,m) = \hat{\mathscr E}(m)/\ZZZ \cong \gre{j_m}.
$$
 \hfill$\triangle$
\end{list}
With these two claims the Theorem will follow, because then we have
\begin{align*}
  \Sal(\A^{\upharpoonright})/G &= \colim_{\ZZZ} \mathscr S = \colim
  \hat{\mathscr G} \\
  &=\colim_{\Fc(\A)} \colim_{\ZZZ} \hat{\mathscr G} = \colim
  \gre{\mathscr D} \simeq \hocolim \gre{\mathscr D},
\end{align*}
where the last homotopy equivalence is given by the Projection Lemma
\cite[Proposition 3.1]{WZZ}.
\end{proof}

\begin{cor}
  $M(\A) \simeq \gre{\GC \mathscr D}$.
\end{cor}
\begin{proof}
  Immediate with the Theorem above and Lemma \ref{lem:Thomason}.
\end{proof}
This prompt us to deviate from
the conventions of \cite{dD2013} and to define the Salvetti complex of a
complexified toric arrangement as follows.

\begin{df} For every complexified toric arrangement $\A$ let
  $$\Sc(\A):= \GC \mathscr D; \quad\quad \Sal (\A):=\gre{\Sc(\A)}. $$
\end{df}

\begin{rem} \label{rem:proj}
  We have immediately $\Sal(\A) \simeq M(\A)$.
Moreover, with  Remark \ref{rem:projGC} we have a cellular map
$$ \pi: \Sal(\A) \to \gre{\Fc(\A)} \simeq T_c$$
induced by the canonical projection from $\hocolim \gre{\mathscr D}$.
\end{rem} 

\subsection{Inclusions}\label{ss:inclusions}

The goal of this section will be to associate to every layer $L$ a
subcomplex of $\Sal(\A)$ homotopy equivalent to the product of $L$
times the complement of the (hyperplane) arrangement $\Ar{L}$. We will
do this in a way that is compatible with the projection to the compact
torus and so that the maps induced in cohomology by the inclusions of
these subcomplexes satisfy a Brieskorn-type compatibility condition
which will be the stepping stone towards a presentation of the
cohomology algebra.

\begin{df}
Given a layer $L \in \Cc$ we write $L_c$ for the intersection $L
\cap T_c$ of $L$ with the compact torus. 
\end{df}
\begin{df}
Let $\A^L$ denote the complexified toric arrangement
defined in the torus $L$ by all hypertori not containing $L$, that is 
the arrangement of all hypertori appearing as a connected component
of an intersection $L\cap K$ for $K\in \A$, $L\not\subseteq K$.	
\end{df}
Notice that the
cellularization $\vert \Fc(\A)\vert$ of $T_c$ restricts to a
cellularization $\gre{\Fc(\A^L)}$ of $L_c$ (i.e., there is a
cellular 
homeomorphism $h: T_c \to \gre{\Fc(\A)}$ with 
$\gre{\Fc(\A^L)}=\gre{\Fc(\A)}\cap
h(L)$).

\begin{thm}\label{thm:inclusion} 
For every layer $L \in \Cc$ and every chamber $B_0\in \Tc(\A_0)$
adjacent to $L_0:=\cap\Ar{L} \in \Ll(\A_0)$, let $F_0:=\overline{B_0}
\cap L_0$. 
Then, there is a subcomplex $\Sl{F_0}$ of the
toric Salvetti complex $\Sal(\A)$ satisfying
\begin{itemize}
\item[(1)] under the canonical projection $\Sal(\A) \to T_c$,
  $\Sl{F_0}$ maps to the layer $L_c$; 
\item[(2)] there is a homotopy equivalence
$\Theta_{F_0}: \Sl{F_0} \to \gre{F(\A^L)} \times \Sal(\Ar{L}) $;
\item[(3)] the first component of $\Theta_{F_0}$ is the projection
  from (1).
\end{itemize}
\end{thm}

We keep the notations of the theorem's claim ($F_0$, $B_0$, $L_0$) split the proof in multiple steps for easier understanding and
later reference.

\begin{df} Recall Definition
    \ref{df:A0}: $F_0$ is a face of $\A_0$ and we can define 
  $$B(F_0):=\{C\in \Tc(\A_0) \mid C\geq F_0 \textrm{ in }\Fc(\A_0)\},$$
  the set of all chambers of $\A_0$ that are adjacent to $F_0$.
 Moreover, for every $F\in\Fc(\A^L)$ define
$$\mu_F:\Tc(\A_0) \to \Tc(\Ar{F}); \mu_F(C) \supseteq C,  $$
\end{df}

\begin{lem} \ \label{lem:duecose}
  \begin{itemize}
  \item[(a)]
    For every morphism $m:F\to G$ in $\mathcal F(\A)$
    and every $C\in \mathcal T(\A_0)$, 
    \begin{equation*}
      S(\mu_F(C),i_m(\mu_G(C))) \cap \Ar{G} =\emptyset.
    \end{equation*}
    \item[(b)]
    For all $F\in \Fc(\A^L)$ we have $F_0\in\Fc(\Ar{F})$, and
    \begin{equation*}
      \mu_F(B(F_0)) = \{C\in \Tc(\Ar{F}) \mid C \geq F_0 \textrm{ in }
      \Fc(\Ar{F})\}.
    \end{equation*}
  \end{itemize}
\end{lem}
\begin{proof}
  For part (a) notice that $C\subseteq \mu_F(C)\subseteq \mu_G(C)$,
  thus for $H\in \Ar{G}$ clearly $\sv{\mu_G(C)}(H) =
  \sv{\mu_F(C)}(H)$, and moreover with Remark \ref{rem:im} we have
  $\sv{i_m(\mu_G(C))}(H) = \sv{\mu_G(C)}(H)$ whenever $H\in \Ar{G}$.
  For part (b) notice first that $\mu_F$ maps chambers to chambers,
  thus it is enough to check that for $C\in B(F_0)$ we have
  $\mu_F(C)\geq F_0$. But the definition of $\mu_F$ is that
  $\sv{\mu_F(C)}(H) = \sv{C}(H)$ for all $H\in \Ar{F}$, thus $F_0\leq
  C$ in $\A_0$ implies $F_0\leq \mu_F(C)$ in $\Ar{F}$.
\end{proof}

\begin{df}
    We now define the following
    diagram: 

  \begin{align*}
    \mathscr D_{F_0}: \Fc(\A^L)^{op} \to &\scwols \\
    F\mapsto & \bigcup_{ B\in B(F_0) } \MS(\Ar{F})_{\mu_F(B)} 
  \end{align*}
and the maps are defined as restrictions of the corresponding maps of
the diagram $\mathscr D$.
\end{df}

\begin{lem}
  The diagram is well-defined, and
$$\mathscr D_{F_0} (F) = \Sc^{F_0}(\Ar{F}).$$
\end{lem}
\begin{proof}
  The diagram is is well-defined because, by Lemma \ref{lem:duecose}.(a) and
\cite[Remark 5.13]{dD2013}, for every $m:F\to G$ in $\mathcal F(\A^L)$
and every $C\in \mathcal T(\A_0)$ the inclusion $j_m:\Sc(\Ar{G})\to \Sc(\Ar{F})$ restricts to an
inclusion $\Sc_{\mu_F(C)} \to \Sc_{\mu_G(C)}$ (compare \cite[Lemma
5.12]{dD2013}).
The second claim follows from Lemma \ref{lem:duecose}.(b).
\end{proof}

\begin{df}
  Define $$\MS_{F_0}:= \gre{\GC \mathscr D_{F_0}}.$$ 
\end{df}

\begin{rem}\label{rem:subc}
  Since $\GC \mathscr D_{F_0}$ is a subcategory of $\GC
    \mathscr D$, $ \MS_{F_0}$ is a subcomplex of $\Sal(\A)$.
\end{rem}

\begin{notation}
  We will form now on use a `column' notation for the Grothendieck
  construction. For a diagram $\mathscr D: \mathcal I \to \scwols$ we
  will write $\colobj{x}{i}$ for the object of $\GC \mathscr D$
  associated to $i\in \operatorname{Ob}\mathcal I$ and $x\in \operatorname{Ob}\mathscr D(i)$,
  and
  \begin{center}
    $\colobj{x_1}{i_1}\colmor{\mu}{f}\colobj{x_2}{i_2}$
  \end{center}
for the
  morphism corresponding to  $f\in\Mor_{\Ic}(i_1,i_2)$ and $\mu \in \Mor_{\mathscr
  D(i_2)}(\mathscr D (f)(x_1),x_2)$
\end{notation}

\begin{lem}
   The canonical projection $\pi: \Sal(\A)
  \to \gre{\Fc(\A)}$ restricts to $\pi_L: \MS_{F_0} \to \gre{\Fc(\A^L)}$.
\end{lem}
\begin{proof}
  This is a check of the definitions, e.g. with Remark \ref{rem:subc}.
\end{proof}

\begin{df}
  Let $\mathscr K_L$ be the constant diagram
  \begin{align*}
    \mathscr K_{L}: \Fc(\A^L)^{op} \to &\scwols\\
    F\mapsto & \Sc(\Ar{L}) \\
 m \mapsto &  \textrm{id}.
  \end{align*}
\end{df}

\begin{df}\label{def:jj}
  Given $F\in \Fc(\A^L)$ let $$\xi{{[F]}}_{F_0}: \SF{F_0}(\Ar{F}) \to
\Sc(\Ar{F}_{\vert F_0\vert = L}) =\Sc(\Ar{L})$$
denote the map described in Definition \ref{def:j} referred to the
`ambient' arrangement $\Ar{F}$.
\end{df}

\begin{lem}
  The maps $\xi[F]_{F_0}$ of Definition \ref{def:jj} induce a natural
  transformation
  $$\mathscr D_{F_0} \Rightarrow \mathscr K_L$$
  and thus a functor
  $$ \Xi_{F_0}: \GC \mathscr D_{F_0} \to \GC \mathscr K_L $$
  which, moreover, induces homotopy equivalence of nerves.
\end{lem}

\begin{proof}
In order to check that the diagram
$$
  \xymatrix{
  \mathscr D_{F_0}(F)  \ar[r]^{\xi{\scriptstyle[F]}_{F_0}} & \Sc(\Ar{L})\\
  \mathscr D_{F_0}(G) \ar[u]^{j_m} \ar[r]_{\xi{\scriptstyle[G]}_{F_0}} & \Sc(\Ar{L})
  \ar[u]_{ = } }
$$
commutes it is enough to see that, for every $K\in \Fc(\Ar{G})$,
$ i_m(K)_{\vert F_0\vert }=K_{\vert F_0\vert}$, as is proved in Remark
\ref{rem:im}. 
Thus, the maps $\xi{[F]}_{F_0}$ induce a well-defined natural
transformation, and thereby the required functor 
$\Xi_{F_0}$, acting on objects $\colobj{K}{F}$ and morphisms $\colobj{\geq}{m}$
(notice that every $\mathscr D_{F_0}(F)$ is indeed a poset), as
$$
\Xi_{F_0} \colobj{K}{F} = \colobj{\xi{{[F]}_{F_0}}(K)}{F}, \quad\quad
\Xi_{F_0} \colobj{\geq_{\,} \\ \,}{m^{\,}}= \colobj{\geq \\ \,}{m}.
$$

Since each map $\xi[F]_{F_0}$ is a homotopy equivalence (by Lemma
\ref{lem:fl} via Quillen's Theorem A \cite{Quillen}), the homotopy
theorem \cite[Proposition 2.3]{WZZ} ensures that the natural
transformation induces homotopy equivalence between homotopy colimits,
thus also between the Grothendieck constructions, as claimed.
\end{proof}

We consider nerves as simplicial sets, and thus denote cells in the
geometric realization of a category by the corresponding chain of morphisms.

\begin{lem}
  A cell of $\MS_{F_0}$ has the form
$$
\sigma= \colobj{D_0}{G_0} \colmor{\geq}{m_1} \colobj{D_1}{G_1} \cdots \colobj{D_k}{G_k}
$$
where $G_i\in \Ob(\Fc(\A^l))$, $m_i: G_{i-1}\to G_i \in
\Mor(\Fc(\A^L))$, $D_i\in \Sc(\Ar{G_i})$ and $D_{i-1} \geq
j_{m_i}(D_i)$. 

Then the function mapping $\sigma$ to 
$$
\sigma \mapsto (G_0 \stackrel{m_1}{\longrightarrow} G_1 \cdots
\stackrel{m_k}{\longrightarrow} G_k, \widehat {\xi_{F_0} D_0} \geq \ldots \geq
\widehat {\xi_{F_0} D_k}) 
$$ 
(where for $D\in \Sc(\Ar{G_i})$ we write $\widehat D := j_{m_i \circ \cdots \circ m_1 } (D_i) \in \Sc(\Ar{G_0})$) induces a homotopy equivalence $\Theta_{F_0}: \MS_{F_0} \to \gre{\Fc(\A^L)} \times \gre{\Sc(\A[L])}$.  \end{lem}

\begin{proof}
  Notice that there is an evident equivalence of categories $\GC
  \mathscr K_L \cong \Fc(\A^L)^{op} \times \Sc(\A[L])$ - thus we see that
$$\Theta_{F_0} = \Sigma_L \circ \vert \Xi_{F_0} \vert,$$ 
the composition of the homotopy equivalence induced by
  $\Xi_{F_0}$ and the canonical (``reverse-subdivision'', see e.g.\ \cite[4.2.2]{Kozlov})
  homeomorphism $\Sigma_L: \gre{\Fc(\A^L)^{op} \times \Sc(\A[L])}\to \gre{\Fc(\A^L)} \times \gre{\Sc(\A[L])} $.
\end{proof}

\begin{df}\label{df:choice}
We now fix for every layer $L$ a face $F_0=F_0(L)$ of $\A_0$ with
$\vert F_0 \vert = L$ and a chamber $B_0$ of $\A_0$ adjacent to $F_0$, and define
$$
\mathscr D_L := \mathscr D_{F_0},\quad\Sl{L} :=\Sl{F_0}, 
\quad \Theta_L:=\Theta_{F_0},
\quad \Xi_L:=\Xi_{F_0}
$$

We call $\varphi_L$ the inclusion map $\Sl{L} \into \Sal(\A)$ from
Theorem \ref{thm:inclusion}. 
\end{df}

Our next goal is the following theorem, which will justify the idea of coherent element given in Definition \ref{def:coherent}.

\begin{thm} \label{lem:coherent}
Fix an integer $q$ and let $L$ be a layer with $\rk(L) > q$. Consider the set 
$(\Cc_{\leq L})_q$  of all the layers $L'$ such that 
$L \subseteq L'$ and
$q = \rk(L')$. 
The following diagram of 
groups is commutative.
 $$
 \xymatrix{
  H^*(\Sal(\A);\Z)\ar[drr]_{\varphi_L^*} \ar[rr]^(.35){\underset{L' \in (\Cc_{\leq L})_q}{\bigoplus}\!\!\!\!\!\varphi_{L'}^*}
  && \underset{L' \in (\Cc_{\leq L})_q}{\bigoplus}
\!\!\!H^q(M(\Ar{L'});\Z) \otimes H^*(L';\Z) 
\ar[d]^(.6){ \underset{L' \in (\Cc_{\leq L})_q}{\sum}\!\!\!\!\!\Lc{(L' \leq L)}} 
\\
&& 
H^q(M(\Ar{L});\Z) \otimes H^*(L;\Z)
 }
 $$
\end{thm}

\begin{lem}\label{CoCaMa}
  The map
$$
C_L: \Fc(\vert \Fc(\A^L) \vert \times \vert \Sc(\Ar{L})\vert) \to \gre{\GC
  \mathscr D_L} , x \mapsto \Theta^{-1}_L(x)
$$
is a contractible carrier map (in the sense of \cite[Chapter II]{LW}).

\end{lem}

\begin{proof}
  Let us consider a cell $x$ as in the claim, say 
$$ x = \gre{(F_k\mk{k-1} F_{k-1} \cdots \mk{1} F_1)} \times \gre{(S_l \geq \cdots
\geq S_1)}. $$ Since $x$ is a
  cell in a product of two regular trisps (see \cite[p.\ 153]{Kozlov}), its
  embedding is homeomorphic to a closed ball. Also, since
  $\Theta_L=\Sigma_L\circ \gre{\Xi_L}$, we first compute the subcomplex
  $\Sigma_L^{-1}(x)$. This is the union of all cells that
  triangulate $x$, hence it consists (see e.g.\ \cite[p.\ 169]{Kozlov}) of  the subcomplex generated by the
  following union of cells (i.e., these cells and every cell in their boundaries).
  $$
  \Sigma_L^{-1}(x) = 
  \bigcup_{f,g} \big\vert(F_{f(r)},S_{g(r)}) 
  \stackrel{(n(f)_{r-1},\geq)}{\longrightarrow}
  \cdots 
  (F_{f(1)},S_{g(1)})
  \big\vert \subseteq \gre{\Fc(\A^{L})^{op}\times \Sc(\Ar{L})}
  $$
where $r:=k+l$ and the union ranges over all 
pairs $(f,g)$, where

\eqnum\label{fgn} $	\left\{\begin{tabular}[l]{l}
 $f:[r]\to [k]$  and 
 $g:[r]\to [l]$ are order-preserving surjections,\\
 and the morphisms $n(f)_{i}$ are defined
 as
$
n(f)_i=\left\{\begin{array}{ll}
m_{f(i)} & \textrm{if } f(i+1)\neq f(i), \\
\id & \textrm{else.}
\end{array}\right.
$
\end{tabular}\right.$

This subcomplex is then, by construction, a triangulation of a closed ball. The preimage
under $\gre{\Xi_L}$ of $\Sigma_L^{-1}(x)$ can be seen as being
covered by subcomplexes of the form $\gre{\Xi_L}^{-1}(y)$, where
$y$ is a cell of $\Sigma_L^{-1}(x)$. The face poset of
the complex $\Sigma_L^{-1}(x)$ is the nerve of this covering,
and thus by the generalized Nerve Lemma \cite[Theorem 15.24]{Kozlov}
$\Theta_L^{-1}(x)$ has the homotopy type of the homotopy
colimit of the associated nerve diagram \cite[15.4.1]{Kozlov}

$$
\mathscr N_{x}: \Fc(\Sigma_L^{-1}(x)) \to \tops, \quad
y \mapsto \gre{\Xi_L}^{-1}(y)
$$
with maps being inclusions. Now we claim that it is enough to prove
that the spaces of $\mathscr N_{x}$ are contractible. Indeed, in
that case the diagram maps of $\mathscr N_{x}$ will be inclusions
of contractible subcomplexes into contractible complexes, and thus in
particular they will be homotopy equivalences. The quasifibration
lemma \cite[Proposition 3.6]{WZZ} then applies and, because the index
category of $\mathscr N_x$ is
contractible (it is the face poset of a triangulation of a closed
ball), will say that $\hocolim \mathscr N_x$ has the same homotopy
type of any of the spaces $\mathscr N_x(y)$ - and hence will
be contractible as required.

To conclude the proof it is thus enough to prove the following.
\begin{itemize}
\item[] {\bf Claim.} For any $y\in \Fc(\Sigma_L^{-1}(x))$,
  the complex $\gre{\Xi_L}^{-1}(y)$ is contractible.
\item[] {\em Proof.} Fix such an $y$, say $y=\vert \sigma \vert$ for a chain
$$
\sigma = (F_k, S_k) \stackrel{(m_{k-1},\geq)}{\longrightarrow} \cdots (F_1, S_1)
$$
in $\Fc(\A)^{op}\times \Sc(\Ar{L})$.
The cells of $\gre{\GC \mathscr D_L}$ which map to $y$ under
$\gre{\Xi_L}$ are all and only those of the form
$$
\colobj{\widetilde S_k}{F_k} \colmor{\geq}{m_{k-1}} \ldots
\colobj{\widetilde S_1}{F_1}
$$
with $\widetilde S _j \in \xi[{F_{j}}]_{F_0}^{-1}(S_j)$. Now, those are exactly
the cells that make up the trisp $\gre{\GC \mathscr G}$, where
$$
\mathscr G: \sigma^{op} \to \scwols,\quad \colobj{S_i}{F_{i}} \mapsto
\xi[F_{i}]^{-1}_{F_0}(S_i), \quad m_i\to j_{m_i}
$$
now, $\gre{\GC \mathscr G}$ is contractible because by Corollary \ref{cor:pre-cont}
the spaces are contractible, hence the diagram
morphisms are homotopy equivalences (inclusions of con\-tracti\-ble
subcomplexes in contractible subcomplexes) and - since the index
category is contractible - the homotopy colimit of $\gre{\mathscr G}$
is contractible. This complete the proof of our claim, hence the
Lemma. $\triangle$
\end{itemize}

\end{proof}

\begin{scholium}\label{ScillaECariddi}
  Given a cell $x=\gre{ F_k\mk{k-1} F_{K-1} \cdots \mk{1} F_1} \times\gre{ S_l \geq \cdots
\geq S_1}$ , the explicit expression for $C_L(x)$ is as the
subcomplex consisting of the following cells
$$
\bigg\vert \colobj{\widetilde S_{g(r)}}{F_{f(r)}}\colmor{
\geq
}{n(f)_{r-1}}
  \cdots 
  \colobj{\widetilde S_{g(1)}}{F_{f(1)}}
\bigg\vert
$$
where $\widetilde S_{g(j)} \in \xi[F_{f(j)}]_{F_0}^{-1}(S_{g(j)})$ and $f,g,n(f)_i$ are
defined in \eqref{fgn}.
\end{scholium}

\begin{proof}[Proof of Theorem \ref{lem:coherent}]
We consider the following diagram.

$$
 \xymatrix{& 
   \ar[dl]^{\varphi_{L'}}\Sc_{L'} = \gre{\GC \mathscr D_{L'}}\ar[r]_(.43){homeq}^(.43){\Theta_{L'}} 
   &
   \gre{\Fc(\A^{L'})} \times \gre{\Sc(\Ar{L'})} \\
   \Sal(\A) 
   &
   &
   \ar[u]^{\iota \times \id}
   \gre{\Fc(\A^{L})} \times \gre{\Sc(\Ar{L'})}
   \ar[d]^{\id \times j_{F'_0}}\\
   & 
   \ar[ul]^{\varphi_{L}}
   \Sc_{L} = \gre{\GC \mathscr D_{L}}\ar[r]^(.43){\Theta_{L}}_(.43){homeq} 
   &
   \gre{\Fc(\A^{L})} \times \gre{\Sc(\Ar{L})} \\
}
 $$
where $F_0'$ is the face associated to $L'$ (Definition \ref{df:choice}) and
$j_{F_0'}$ is as in Definition \ref{df:jb}

Now, using the Carrier Lemma \cite[Proposition II.9.2]{LW} and Lemma \ref{CoCaMa} choose maps
$J_L$ and $J_{L'}$ carried by  $C_L$ resp.\ $C_{L'}$. The explicit form
of $C_L$ given before shows that both carrier maps $\Theta_L \circ
C_L$ and $C_L \circ \Theta_L$ carry the identity map. This proves -
again, by the Carrier Lemma, that $J_L$ is a homotopy inverse to
$\Theta_L$ (and similarly for $J_{L'}$ and $\Theta_{L'}$). Consider
then the following diagram.

\begin{equation} \label{eq:commute}
\begin{split}
\xymatrix{& 
   \ar[dl]^{\varphi_{L'}}\Sc_{L'} = \gre{\GC \mathscr D_{L'}}
   &
   \ar[l]^(.57){J_{L'}}_(.57){homeq}
   \gre{\Fc(\A^{L'})} \times \gre{\Sc(\Ar{L'})} \\
   \Sal(\A) 
   &
   &
   \ar[u]^{\iota \times \id}
   \gre{\Fc(\A^{L})} \times \gre{\Sc(\Ar{L'})}
   \ar[d]^{\id \times j_{F_0'}}\\
   & 
   \ar[ul]^{\varphi_{L}}
   \Sc_{L} = \gre{\GC \mathscr D_{L}} 
   &
   \ar[l]^(.57){J_{L}}_(.57){homeq}
   \gre{\Fc(\A^{L})} \times \gre{\Sc(\Ar{L})} \\
}
\end{split}
\end{equation}

Again by Lemma \ref{CoCaMa} the map $C_{L'}\circ (\iota \times \id)$ is a
contractible carrier map which carries $J_{L'}\circ (\iota \times
\id)$. In the next claim we will prove that the same carrier map
carries also  $J_L\circ (\id\times j_{F_0'})$. By the contractible
carrier Lemma \cite[Proposition II.9.2]{LW} then the two compositions in the diagrams are
homotopic - thus in particular the diagram commutes in cohomology.
\begin{itemize}
\item[] {\bf Claim.} $C_{L'}\circ (\iota \times \id)$ carries $J_L\circ (\id\times j_{F_0}).$
\item[] {\em Proof.} We know that $C_L\circ (\id\times j_{F_0})$
  carries $J_L\circ (\id\times j_{F_0})$. It is then enough to prove
  that, for every cell $\tau$ of $\gre{\Fc (\A^L)} \times
  \gre{\Sc(\Ar{L'})}$,
  \begin{equation}
C_L((\id\times j_{F_0'}) (\gre{\tau})) \subseteq
  C_{L'}((\iota \times \id). \label{eq:3}
\end{equation}
To this end, we refer to the explicit description of these complexes
given in Scholium \ref{ScillaECariddi} and see that, for a given $\tau$, the
composable chains indexing maximal cells of the two subcomplexes are
completely determined by the `$S$-components' of the objects - since
the $F$-components and the morphisms are completely determined by the
$F$-component of $\tau$ and the fact that $\Sc(\Ar{L})$ and
$\Sc(\Ar{L'})$ are posets.

Thus to prove \eqref{eq:3} it is enough to prove that, for every cell
$S$ of $\Sc(\Ar{L'})$ and every $F\in \Fc(\A^L)$,
  \begin{equation}
    \xi[F]^{-1}_{F_0}(j_{F_0'} (S)) \subseteq \xi[F]^{-1}_{F_0'} (S).
 \label{eq:4}
\end{equation}
This is now a computation. Write $S=[G,K]$ and recall the definition
of $j_{F_0'}$ (Definition \ref{df:jb}). Then $j_{F_0'}([G,K]) = [i_{F_0'}(G),i_{F_0'}(K)]$ and,
with the expression given in Lemma \ref{lem:preim}, we need to verify
$$
\xi[F]_{F_0'}(\{[R,i_{F_0'}(K)_R] \mid R_{\vert F_0 \vert} =
i_{F_0'}(G)\})
=\{ [G,K]\}.
$$
Now, if $R_{\vert F_0 \vert} =
i_{F_0'}(G)$, then
$$
R_{\vert F_0'\vert} = (R_{\vert F_0\vert})_{\vert F_0' \vert} =
(i_{F_0'}(G) )_{\vert F_0' \vert} = G
$$
(the first equality because $\vert F_0 \vert \subseteq \vert F_0'\vert$, the second by assumption,
the third by Lemma \ref{piccolezze}.(1)) and the condition on $R$ is also satisfied. 
So 
we only need to show that
$(i_{F_0'}(K)_R)_{\vert F_0' \vert}=K.$ 

Again, we compute
$$
\sv{(i_{F_0'}(K)_R)_{\vert F_0' \vert}} (H)=
\left\{\begin{array}{ll}
  \sv{R}(H) & \textrm{ if } \sv{F'_0}(H)=0, \sv{R}(H)\neq 0 \\
  \sv{i_{F'_0}(K)}(H) = \sv{K} (H) & \textrm{ if } \sv{F'_0}(H)=0, \sv{R}(H) =0 
\end{array}\right.
$$ 
Now consider the second alternative and remember that $R_{\vert F_0' \vert }=G$, thus when $\sv{F'_0}(H)=0$ we
have $\sv{R}(H)=\sv{G}(H) \leq \sv{K}(H)$ and if additionally $\sv{R}(H)\neq 0$, $\sv{R}(H)=\sv{G}(H)=\sv{K}(H)$,
as required. Hence the claim follows.
\end{itemize}

In order to complete the proof of Theorem \ref{lem:coherent} we can consider the cohomology diagram corresponding to diagram \eqref{eq:commute}.

\begin{equation} \label{eq:commutecohom}
\begin{split}
\xymatrix{
					& & & H^*(M(\Ar{L'});\Z) \otimes H^*(L';\Z) \ar[d]^{\id \otimes \iota^*}\\
H^*(\Sal(\A) ;\Z) \ar[rrru]^(.40){\varphi_{L'}^*} \ar[rrrd]_{\varphi_L^*}                  & & & H^*(M(\Ar{L'});\Z) \otimes H^*(L;\Z)  \\
                  & & & H^*(M(\Ar{L});\Z) \otimes H^*(L;\Z) \ar[u]_{j_{F_0'}^* \otimes \id} \\ 
                  }
\end{split}
\end{equation}

We project to the cohomological degree $q$ for the first factor of the tensor product for the terms in the right side of diagram \eqref{eq:commutecohom} and we take the direct sum for all $L' \in \Cc_{\leq L})_q.$ 
We get:
\begin{equation*} \label{eq:commutecohom2}
\begin{split}
\xymatrix{
& & & \underset{L' \in (\Cc_{\leq L})_q}{\bigoplus}
\!\!\! H^q(M(\Ar{L'});\Z) \otimes H^*(L';\Z) \ar[d]^{\id \otimes \iota^*}\\
H^*(\Sal(\A) ;\Z) \ar[rrru]^(.40){\underset{L' \in (\Cc_{\leq L})_q}{\bigoplus}
\!\!\! \varphi_{L'}^*} \ar[rrrd]_{\varphi_L^*}                  
& & & \left( \underset{L' \in (\Cc_{\leq L})_q}{\bigoplus}
\!\!\! H^q(M(\Ar{L'});\Z) \otimes \right) H^*(L;\Z)  \\
& & & H^q(M(\Ar{L});\Z) \otimes H^*(L;\Z) \ar[u]_{\left( \underset{L' \in (\Cc_{\leq L})_q}{\bigoplus}
\!\!\! j_{F_0'}^* \right) \otimes \id}. \\ 
}\end{split}
\end{equation*}
From the Combinatorial Brieskorn Lemma (Proposition \ref{cor:cbl}) the left factor $\underset{L' \in (\Cc_{\leq L})_q}{\bigoplus}
\!\!\! j_{F_0'}^* $ of the bottom right map
in an isomorphism. 
Hence we can invert the bottom right arrow and, composing with the upper right map $\id \otimes \iota^*$ we get $\sum_{L' \in (\Cc_{\leq L})_q}\Lc{(L' \leq L)}$. Thus, the commutativity of the diagram of the statement of the Theorem follows.
\end{proof}

\section{The complexified case} \label{sec:main}

In this section we provide an explicit description of the cohomology ring of the complement of a complexified toric arrangement. To this end we will consider for all $L \in \C$ the inclusion $\varphi_L:\Sc_L \to \Sal(\A)$ from Theorem \ref{thm:inclusion} and use these maps in order to define a map $$\Phi: \bigsqcup_{L \in \C} \Sc_L \to \Sal(\A).$$ 
The understanding of the corresponding cohomology homomorphism $\Phi^*$ is the key ingredient for the proof of Theorem \ref{thm:main} in the complexified case. 

\subsection{Spectral Sequences} \label{sec:lss} \label{sec:SpSe}

We consider the Leray spectral sequence $\widehat E_*^{p,q}$ induced by the projection $\pi: \Sal(\A) \to T_c$ (defined as in Remark \ref{rem:proj}) with second page
$$
\widehat E_2^{p,q} = H^p(T_c;\mathscr H^q (\pi;\Z )),
$$
where $\mathscr H^q (\pi;\Z )$ is the sheaf given by the sheafification of the presheaf
$$
U \mapsto H^q(\pi^{-1}(U);\Z).
$$

\begin{rem}
The spectral sequence above is equivalent to the one used by Bibby in \cite{bibby2013}, induced
by the inclusion $M(\A) \into T$. In fact the inclusions $T_c \into T$ and $\Sal(\A) \into M(\A)$ are homotopy equivalences and the following square is homotopy commutative.
$$\xymatrix{
\Sal(\A)\ar[r]^{\subset} \ar[d]^{\pi} & M(\A) \ar[d]^{\subset} \\
T_c \ar[r]^{\subset} & T
}$$
Moreover for every point $z \in T_c$, let $L_z := \bigcap_{z \in L \subset \Cc}L$  be the unique layer containing $z$ in its interior and let $\Cc_{(z)}:= \{L \in \Cc \mid z \in L\}= \Cc_{\leq L_z}$ be the set of layers containing $z$. Then, given an open set $U \in T$ containing $z$, if $U$ is small enough then $\pi^{-1}(U \cap T_c)$ is homotopy equivalent to $\Sal(\Ar{L_z})$
and the inclusion $\pi^{-1}(U \cap T_c) \into U \cap M(\A)$ is an homotopy equivalence.
\end{rem}

Given a layer $L \in \Cc$, let $\pi_L$ be the restriction of the map $\pi$ to the subcomplex $\Sl{L}$: $\pi_L: \Sl{L} \to T_c$. Let $_LE_*^{p,q}$ be the Leray spectral induced by the projection $\pi_L$, with second page
$$
_LE_2^{p,q} = H^p(T_c;\mathscr H^q (\pi_L;\Z )).
$$
We notice that the sheaf $\mathscr H^q (\pi_L;\Z )$ is supported on the subtorus $L_c = L \cap T_c$.

\begin{lem}  \label{lem:decompss} 
The sheaf $\mathscr H^q (\pi;\Z )$ is a direct sum of the sheaves
$$
\bigoplus_{L \in \Cc_q} [H^q(M(\Ar{L});\Z)]_{L_c}
$$
where $[H^q(M(\Ar{L});\Z)]_{L_c}$ is the restriction of the constant sheaf 
$H^q(M(\Ar{L});\Z)$ to the layer $L_c$. 
We have the following decomposition
\begin{equation}\label{eq:dec1}
\widehat E_2^{p,q} = 
\bigoplus_{L' \in \Cc_q} 
H^p(L';\Z)\otimes H^q(M(\Ar{L'});\Z) 
\end{equation}
and for every layer $L \in \Cc$
\begin{equation} \label{eq:decL}
_{L} E_2^{p,q} = 
\bigoplus_{L'' \in (\Cc_{\leq L})_q} 
H^p(L;\Z)\otimes H^q(M(\Ar{L''});\Z).
\end{equation}
\end{lem}
\begin{proof} We prove our lemma with a straightforward generalization of the argument 
given in \cite[Lemma 3.1]{bibby2013}.

Let $L'$ be a layer in $\Cc_q$. 
The set of subtori of $\A$ that contain $L'$ is 
in bijection with the central arrangement $\Ar{L'}$. 
The projection $\pi_{L'}:\Sl{L'} \to T_c$ defines the sheaf $\epsilon_{L'} := \mathscr H^q (\pi_{L'};\Z )$ on $T_c$. The sheaf $\epsilon_{L'}$ has support in $L'_c$ and the stalk at a point $z \in L'_c$ is $(\epsilon_{L'})_z \simeq H^q(\Sal(\Ar{L'});\Z)$. Moreover, since the complex $\Sl{L'}$ is homotopy equivalent to the product $L'_c \times \Sal(\Ar{L'})$, the sheaf $\epsilon_{L'}$ is constant on $L'_c$.

Then we can consider the sheaf $\epsilon := \bigoplus_{L' \in \Cc_q} \epsilon_{L'}$. The stalk at $z \in T$ is 
$$
\epsilon_z = \bigoplus_{L' \in \Cc_q} (\epsilon_{L'})_z \simeq \bigoplus_{L' \in \Cc_q} H^q(\Sal(\Ar{L'});\Z)
$$
and 
the map
$$
\epsilon_{L'} \to \mathscr H^q (\pi;\Z )
$$
induced by inclusion $L'_c \into T_c$ and hence by the commutative diagram
$$
\xymatrix{
\Sl{L'} \ar[d]^{\pi_{L'}} \ar[r]^{\subset} & \Sal(\A)\ar[d]^{\pi}  \\
L'_c \ar[r]^{\subset} & T_c
}
$$
can be described in terms of Brieskorn Lemma (Proposition
\ref{lem:brieskorn}). In fact 
for every point $z \in L'_c$ we have that the map of stalks
$$(\epsilon_{L'})_z \to \mathscr H^q (\pi;\Z )_z$$ is induced by an inclusion
$\Sal(\Ar{L'}) \into \Sal(\Ar{L_z})$ 

and corresponds to the map 
$$j_F^*: H^q(\Sal(\Ar{L_z});\Z) \to H^q(\Sal(\Ar{L'});\Z)$$ where $j_F$ is the map given in Definition \ref{df:jb}.
From Proposition \ref{cor:cbl} the last map is the left inverse of map
$$b^q: H^q(M(\Ar{L'});\Z) \to H^q(M(\Ar{L_z});\Z)$$ of Definition \ref{def:b}, that is the $\A[L']$-component of the Brieskorn decomposition. Since $(\epsilon)_z$ is given by the direct sum of $(\epsilon_{L'})_z$ over all $L \in \Cc_q$ containing $z$, it follows that
the corresponding map $\epsilon \to \mathscr H^q (\pi;\Z )$ is an isomorphism of sheaves (in fact, this map is the Brieskorn isomorphism).

The first part of the lemma is now straightforward, since we have that
\begin{eqnarray*}
\widehat E_2^{p,q} & = & H^p(T_c; \mathscr H^q (\pi;\Z )) = \\
            & = & H^p(T_c; \epsilon) = \\
            & = & \bigoplus_{L' \in \Cc_q} H^p(T_c; \epsilon_{L'}) = \\
            & = & \bigoplus_{L' \in \Cc_q} H^p(L_c'; \Z) \otimes H^q(M(\Ar{L'});\Z).   
\end{eqnarray*}
The second part of the lemma follows since the subcomplex $\Sl{L}$ 
is homotopy equivalent to the product $L_c \times \Sal(\Ar{L}) $ and the map $\pi_L$ is 
homotopically equivalent to the 
projection on the first factor (see Theorem \ref{thm:inclusion}). 
Hence the sheaf  $\mathscr H^q (\pi_L;\Z )$ is the constant
sheaf $H^*(M(\Ar{L});\Z)$ and the decomposition given in \eqref{eq:decL} follows from the decomposition given by the Brieskorn Lemma applied  to $H^*(M(\Ar{L});\Z)$.
\end{proof}

\begin{thm}\label{thm:collapses}
The spectral sequences $\widehat E_*^{p,q}$ and $_LE_*^{p,q}$ collapse at the second page.
\end{thm}

\begin{proof}
We can prove the collapsing of $\widehat E_2^{p,q}$ by means of a counting argument.
We assume the arrangement $\A$ to be ordered and we define a no broken circuit in $\Cc_{\leq L}$ via the natural poset-isomorphism with $\Ar L$ (see Remark \ref{rem:posetArL}). 
According to De Concini-Procesi \cite{dp2005} (see also Looijenga \cite{looi93})
the Poincar\'e polynomial $P_{\A}(t)$ of the cohomology $H^*(M(\A);\C)$ of a toric arrangement $\A$ in a complex torus $T$ of dimension $d$ is given by
$$
P_{\A}(t) = \sum_{j=0}^{\infty}  |\Nc_j| (1+t)^{d-j}t^j
$$
where we define $$\Nc_j:= \{ (L,N) \in \Cc_j \times \mathscr{P}(\A) \mid N \mbox{ is a no broken circuit set of } \Cc_{\leq L} , |N| =j\}$$
(we refer to \cite[Definition 3.6]{dp2005}) for a definition of \emph{no broken circuit set}).

We compare the Poincar\'e polynomial above with the rank of the term $\widehat E_2^{p,q}$. For a fixed $q$
we have that $\oplus_p \widehat E_2^{p,q}$ is a free $\Z$-module with Poincar\'e polynomial given by
$|\Nc_q|(1+t)^{d-q}.$ Hence the total rank of $\widehat E_2^{p,q}$ is computed by its Poincar\'e polynomial $$\sum_{q \geq 0} |\Nc_q|(1+t)^{d-q}t^q$$ and the spectral sequence must collapse at the page $\widehat E_2.$

The collapsing for $_LE_2^{p,q}$ is straightforward since the projection $\pi_L: \Sl{L} \to T_c$ maps onto the compact subtorus $L_c$ and (again, applying Theorem \ref{thm:inclusion}) the subcomplex $\Sl{L}$ factors as $L_c \times \Sal(\Ar L)$ where $\pi_L$ is the projection on the first factor. Hence the spectral sequence trivially collapses at the second page, since the two factors have torsion-free integer cohomology.
\end{proof}

We will see in Section \ref{sec:combiforse} that a similar argument shows that the Leray spectral sequence induced by the inclusion $M(\A) \in T$ always collapses at the second page.

\begin{rem}
It has already been noticed in \cite{bibby2013} that the analogous spectral sequence over the rationals collapses at the third page in the more general case of smooth connected divisors intersecting like hyperplanes in a smooth complex projective variety.
\end{rem}

\begin{thm}\label{thm:ssmap}
The inclusion $\varphi_L: \Sl{L} \into \Sal(\A)$ induces a natural morphism of spectral sequences
$$
\widehat E_*^{p,q} \stackrel{\overline{\varphi}^*_L}{\longrightarrow}  {}_LE_*^{p,q}. 
$$
The map $\overline{\varphi}_L^*: \widehat E_2^{p,*} \to   {}_LE_2^{p,*}$
is the natural map 
$$
H^p(T;\mathscr H^* (\pi;\Z )) \to H^p(T;\mathscr H^* (\pi_L;\Z ))
$$
induced by the morphism of $\Z$-algebras
$$
H^*(\pi^{-1}(U);\Z) \to H^*(\pi_L^{-1}(U);\Z)
$$
given by the inclusion
$$
\pi_L^{-1}(U) \into \pi^{-1}(U).
$$
\end{thm}
This is obvious from the definition (or following Leray's 
argument in \cite{leray46}, see also \cite{segal68}).

\begin{cor}\label{cor:decompos_ss}
The morphism $$\overline{\varphi}_L^*: 
\bigoplus_{L' \in \Cc_q} 
H^p(L';\Z)\otimes H^q(M(\Ar{L'});\Z) \to 
\bigoplus_{L'' \in (\Cc_{\leq L})_q}  
H^p(L;\Z)\otimes H^q(M(\Ar{L''});\Z)
$$
decomposes in maps between the directs summands as follows. For every $L' \in \Cc_q, L'' \in (\Cc_{\leq L})_q$ we have a map:

\begin{align*}
  H^p(L';\Z)\otimes H^q(M(\Ar{L'});\Z) &\to H^p(L;\Z)\otimes H^q(M(\Ar{L''});\Z) \\
\omega \otimes \lambda &\mapsto \left\{  
\begin{array}{ll}
i^*(\omega) \otimes \lambda& \mbox{ if } L' = L'',\\
0 & \mbox{ otherwise.} 
\end{array}
\right.
\end{align*}
In particular $\overline{\varphi}^*_L$ is a surjective map.
\end{cor}

\begin{proof}
According to Theorem \ref{thm:ssmap} we have the commutative square
$$
\xymatrix{\underset{L' \in \Cc_q}{\bigoplus}
H^p(L';\Z)\otimes H^q(M(\Ar{L'});\Z) \ar[d]^{\simeq} \ar[r]^{\overline{\varphi}_L^*} &
\underset{L'' \in (\Cc_{\leq L})_q}{\bigoplus}  
H^p(L;\Z)\otimes H^q(M(\Ar{L''});\Z) \ar[d]^{\simeq}\\ 
H^p(T_c; \underset{L' \in \Cc_q}{\bigoplus} [H^q(M(\Ar{L'});\Z)]_{L'}) \ar[r] & H^p(T_c;
\underset{L'' \in (\Cc_{\leq L})_q}{\bigoplus}  [H^q(M(\Ar{L''});\Z)]_{L})&
}
$$
where the vertical maps are isomorphisms 
and the map in the bottom row is induced by the natural map of
sheaves. The top arrow is simply inclusion in the corresponding direct
summands. More precisely, the top arrow splits in a direct sum of maps, a null map for every $L' \not\leq L$ and the natural restriction map 
\begin{equation}\label{eq:restrictions}
H^p(T_c; [H^q(M(\Ar{L''});\Z)]_{L''}) \to 
H^p(T_c; [H^q(M(\Ar{L''});\Z)]_{L}),
\end{equation}
for every $L' = L'' \in (\Cc_{\leq L})_q$. The two sheaves that appear in \eqref{eq:restrictions} are two
restrictions of the same constant sheaf, thus the map is the projection
$$
H^p(L'') \otimes H^q(M(\Ar{L''});\Z)) \stackrel{i^* \otimes \mbox{id}}{\longrightarrow} 
H^p(L)\otimes H^q(M(\Ar{L''});\Z)
$$
given by the tensor product of the projection on the first factor induced by the inclusion $i: L\into L''$ and the identity on the second factor.

The surjectivity of $\overline{\varphi}^*_L$ follows from the surjectivity of the cohomology homomorphism $i^*$ induced by the inclusion $i:L \into L'$ for $L' \leq L$.
\end{proof}

\begin{cor} \label{cor:decomp1}With respect to the decomposition given in Equation \eqref{eq:dec1} of Lemma \ref{lem:decompss}, for a fixed $L \in \Cc$ the map $\overline{\varphi}_L^*: \widehat{E}_2^{p,q}\to {}_LE_2^{p,q}$
restricts on every $L'$-summand of $\widehat{E}_2^{p,q}$ as follows
$$
H^p(L') \otimes H^q(M(\Ar{L'});\Z)) \to 
H^p(L; H^q(M(\Ar{L});\Z)) = {}_LE_2^{p,q}
$$

$$
\omega \otimes \lambda \mapsto \left\{  
\begin{array}{ll}
i^*(\omega) \otimes b(\lambda) & \mbox{ if } L' \leq L,\\
0 & \mbox{ otherwise.} 
\end{array}
\right.
$$
\end{cor}
\begin{proof}
The statement follows from Corollary \ref{cor:decompos_ss}, applying the Brieskorn decomposition to Equation \eqref{eq:decL} of Lemma \ref{lem:decompss}.\todo{troppo poco?}
\end{proof}
\begin{rem} \label{rmk:decomp2} Recall the definition of the map $\Phi:\bigsqcup_{L \in \C} \Sc_L \to \Sal(\A)$ given at the beginning of Section \ref{sec:main}. The morphism of spectral sequences, and of algebras,
$$
\overline{\Phi}^*: \widehat E_*^{p,q} \longrightarrow \bigoplus_{L \in \Cc} {}_LE_*^{p,q}
$$
induced by the map $\Phi$ is the direct sum $\oplus_{L \in \Cc} \overline{\varphi}_L^*$.
\end{rem}

\subsection{Algebras} \label{ss:algebras} 
Recall Definition \ref{def:algebra_a} of the algebra $\Aa(\A)$.
In this section we define the algebra $\Ab(\A)$ and we prove that the algebras $\Aa(\A)$ and $\Ab(\A)$ are isomorphic. We will see that two algebras turn our to be both isomorphic to the integer cohomology ring of the complement $M(\A)$ of a toric arrangement. In the case of a real complexified arrangement this is proved in Section \ref{ss:proof_main}. \todo{ho pensato di non mettere l'os\-ser\-va\-zio\-ne seguente come Remark: mi sembra troppo corta e mi suona male un remark ``preventivo''.}Actually the abstract construction of the two algebras and of their isomorphism holds not only for real complexified toric arrangements, but also for the general case, thus the following results will be useful also in Section \ref{sec:general}, the proof of Theorem \ref{thm:final_part_A}.

We begin with some definitions.

\begin{df}\label{df:map_p}
The map of $\Z$-modules $p:\Aa(\A) \to \bigoplus_{L \in \Cc} \Lc{L}$ is the map defined on generators by:
$$
\Lc{L}^{\rk(L)} \ni \alpha \mapsto p(\alpha)$$
where
$$p(\alpha)^q_{L'} = \left\{ \begin{array}{ll}
\Lc{L \leq L'}(\alpha) &\mbox{ if }L \leq L', \\
0 & \mbox{ otherwise.}
\end{array}\right.
$$
\end{df}

\begin{rem}
The map $p$ is a section of the natural projection $\pi: \bigoplus_{L \in \Cc} \Lc{L} \to \Aa(\A)$.
\end{rem}

\begin{prop}\label{prop:iso_mod}
The map $p:\Aa(\A) \to \bigoplus_{L \in \Cc} \Lc{L}$ is injective and its image is the submodule of coherent elements (introduced in Definition \ref{def:coherent}).
\end{prop}
\begin{proof}
It is clear that the image of $p$ is given by coherent elements since the images of all the generators are coherent.
The map $p$ is clearly injective since the projection of $\pi: \bigoplus_{L \in \Cc} \Lc{L}  \to \Aa(\A)$ is a left inverse for $p$.
Finally, since it follows from Definition \ref{def:coherent} that every coherent element $\alpha$ is determined by its projection on $\Aa(\A)$, the map $p$ is surjective on the submodule of coherent elements. 
\end{proof}

We need to introduce a special product on the sum of algebras $\oplus_{L \in \Cc} \Lc{L}$.
\begin{df}\label{def:cdot}
The product $\odot$ on the direct sum $\oplus_{L \in \Cc} \Lc{L}$ is given on generators as follows.
Let $\alpha \in \Lc{L}^q$ and $\alpha' \in \Lc{L'}^{q'}$:$$
(\alpha \odot \alpha')_{L''} := \left\{
\begin{array}{ll}
\Lc{L \leq L''}(\alpha) \cUp \Lc{L' \leq L''}(\alpha') & \mbox{ if }\rk (L) = q, \rk (L') = q' \mbox{ and } L \cap L' \leq L'',\\
0 & \mbox{ otherwise.} 
\end{array}
\right.
$$
\end{df}

\begin{rem} 
Notice that the product $\odot$ restricts on the subgroup $\Aa(\A)$ to the product $\pra$ introduced in Definition \ref{def:algebra_a}.
\end{rem}

\begin{lem}\label{lem:products}
The restriction of the product $\odot$, defined on the sum $\oplus_{L \in \Cc} \Lc{L}$, to the submodule of coherent elements equals the restriction of the natural product given on $\oplus_{L \in \Cc} \Lc{L}$ viewed as a direct sum of cohomology rings.
\end{lem}

\begin{proof}
Since all coherent elements are in the image of $p$, we can easily check the lemma on generators of the form $p(\alpha)$ for $\alpha \in \Lc{L'}^q$.
\end{proof}

The following proposition allows us to define the algebra of coherent elements.

\begin{prop}\label{prop:colim}
The product $\odot$ maps two coherent elements $\alpha \in \bigoplus_{L \in \Cc} \Lc{L}^q$
and $\alpha' \in \bigoplus_{L \in \Cc} \Lc{L}^{q'}$ to a coherent element 
in $\bigoplus_{L \in \Cc_{\geq (q+q')}} \Lc{L}^{q+q'}$. 
\end{prop}
\begin{proof}
From the definition of the product $\odot$ in $\oplus_{L \in \Cc} \Lc{L}$ we have that for
$L'' \in \Cc_{\geq (q+q')}$ 
$$
(\alpha \odot \alpha')_{L''} = \!\!\!\! \sum_{
\substack{
{L \in (\Cc_{\geq L''})_q}\\
{L' \in (\Cc_{\geq L''})_{q'}}
}
} \!\!\!\!  \Lc{L \leq L''}(\alpha_L) \cUp \Lc{L' \leq L''}(\alpha'_{L'}).
$$
We claim that such an element is coherent. For $\rk(L'') = q + q'$ there is nothing to check, 
while for $\rk(L'') > q + q'$ we need to check that
\begin{equation}\label{eq:check1}
(\alpha \odot \alpha')_{L''} = \sum_{\wL \in (\Cc_{\leq L''})_{q+q'}} \Lc{\wL\leq L''}(\alpha \odot \alpha')_{\wL}.
\end{equation}

In order to prove Equation \eqref{eq:check1}, we need to show that if $\alpha \in \Lc{L}^{\rk (L)}$, $\alpha' \in \Lc{L'}^{\rk (L')}$ and $\rk(L) + \rk(L') > \rk (L \cap L')$ then $\alpha \odot \alpha'= 0$. First notice that even if $L \cap L'$ can not be a layer. Nevertheless its connected components are parallel layers with the same rank, hence $\rk (L \cap L')$ is well defined. The equality $\alpha \odot \alpha'= 0$ follows since for every connected component $L^\circ$ of $L \cap L'$ we have $\Lc{L^\circ}^{\rk(L) + \rk(L')}=0$ and this imply that $(\alpha \odot \alpha')_{L^\circ}= 0$. Moreover for any layer $\overline{L}$ such that 
$\overline{L} > L$ and $\overline{L} > L'$ we have that there exist a unique connected component
$L^\circ$ of $L\cap L'$ such that $\overline{L} > L^\circ$ and we have that 
\begin{eqnarray*}
(\alpha \odot \alpha')_{\overline{L}} &=& 
\Lc{L \leq \overline{L}}(\alpha) \cup \Lc{L' \leq \overline{L}}(\alpha')
=\\&=&\Lc{L^\circ \leq \overline{L}}(\Lc{L \leq L^\circ}(\alpha) \cup \Lc{L' \leq L^\circ}(\alpha')) = \Lc{L^\circ \leq \overline{L}}(\alpha \odot \alpha')_{L^\circ}=0.
\end{eqnarray*}
Then we can assume that for any pair of layers $(L,L') \in \Cc_q \times \Cc_{q'}$ we have that
either $\rk (L\cap L') = q+q'$ or $\alpha_L \odot \alpha_{L'}=0$.
Now the equality 
\begin{equation}\label{eq:doublesum}
(\alpha \odot \alpha')_{L''} = \sum_{\wL \in (\Cc_{\geq L''\!})_{(q+q')}\phantom{xx}} 
\!\!\!\! \sum_{
\substack{
L \in (\Cc_{\geq \wL})_q\\
L' \in (\Cc_{\geq \wL})_{q'}
}} \!\!\!\!  \Lc{L \leq L''}(\alpha_L) \cUp \Lc{L' \leq L''}(\alpha'_{L'})
\end{equation}
follows by bilinearity, since it holds when $\alpha, \alpha'$ are supported on a single layer.
From Equation \eqref{eq:doublesum} it follows that Equation \eqref{eq:check1} holds. 
\end{proof}

\begin{df} \label{def:algebra_b} 
Let $\A$ be a toric arrangement in $T$ and let $\Cc$ be the poset of layers.
We define the algebra $\Ab(\A)$ as the subring of
$$
\bigoplus_{L \in \Cc} \Lc{L}
$$
endowed with the natural product $\odot$, generated as a $\Z$-module by the coherent elements.
\end{df}

\begin{rem}
Lemma \ref{lem:products} implies that $\Ab(\A)$ is also a subring of $\bigoplus_{L \in \Cc} \Lc{L}$ with respect of the multiplication given by $\odot$. Moreover, as previously remarked, the two product structures coincide on $\Ab(\A)$.
\end{rem}

\begin{prop}\label{prop:iso_algebras}
The map of $\Z$-modules $p:\Aa(\A) \to \Ab(\A)$ is an isomorphism of algebras.
\end{prop}
\begin{proof}
Given $\alpha, \alpha' \in \Aa(\A)$, we remarked that
$$
\pi( p(\alpha) \odot p(\alpha')) = \pi( p(\alpha) \cup p(\alpha')) =  \alpha \pra \alpha'
$$
where the second equality follows from the definition of the product $\pra$. Hence proposition follows from the injectivity of $\pi:\Ab(\A) \to \Aa(\A)$.
\end{proof}

\subsection{Proof of the main theorem} \label{ss:proof_main}

We are ready to prove our main theorem in the case of a real complexified toric arrangement.

\begin{proof}[Proof of Theorem \ref{thm:main} (complexified case)]
We will consider the map $\Phi: \bigsqcup_{L \in \Cc}\Sl{L} \to M(\A)$ defined at the beginning of Section \ref{sec:main}. This induces a map in cohomology
$$
\Phi^*: H^*(M(\A);\Z) \to \bigoplus_{L \in \Cc}
\Lc{L}
$$
that is a ring homomorphism with respect to the natural product on the sum of cohomology rings $\bigoplus_{L \in \Cc}\Lc{L}$. 

Moreover it follows from Proposition \ref{lem:coherent} that the image of $\Phi^*$ is given by 
coherent elements. Hence $\Phi^*$ maps in $\Ab(\A)$ 
and from Lemma \ref{lem:products} it is also an homomorphism
of rings, with respect to the product $\odot$ defined in Section \ref{ss:algebras}.

We need to prove the injectivity and surjectivity of the map $\Phi^*$.

Recall that from Lemma \ref{lem:decompss} we can write
$$
\widehat E_2^{p,q} = \bigoplus_{L \in \Cc} H^p(T;[H^{\rk(L)}(\Sal(\Ar{L});\Z)]_L)
$$
for the Leray spectral sequence associated to the projection $\pi: \Sal{\A}\to T_c$
and
$$
{}_LE_2^{p,q} = H^p(T_c;\mathscr H^q (\pi_L;\Z )) = H^p(L;\Z) \otimes H^q(M(\Ar{L});\Z)
$$
for the Leray spectral sequence associated to the map $\pi_L: \Sl{L} \to T_c$. Moreover, also due to 
Lemma \ref{lem:decompss}, we have the decomposition
$$
_{L} E_2^{p,q} = 
\bigoplus_{L'' \in (\Cc_{\leq L})_q} 
H^p(L;\Z)\otimes H^q(M(\Ar{L''});\Z).
$$

From Theorem \ref{thm:inclusion} we have that for any layer $L$ we have a commutative square
$$
\xymatrix{
\Sl{L} \ar[d]^{\pi_{L}} \ar[r]^(.35){\subset} & \Sal(\A)\ar[d]^{\pi}  \\
L_c \ar[r]^{\subset} & T_c
}
$$
and hence a map $\overline{\varphi}_L^*$ of Leray spectral sequences in cohomology, together with the 
following commutative diagram, where, from Theorem \ref{thm:collapses} the vertical maps $\overline{\psi}$ and $\psi_L$ are group isomorphisms; moreover $\psi_L$ and the horizontal maps $\varphi^*_L$ and $\overline{\varphi}_L^*$ are ring homomorphisms:\todo{non ho messo il ``cubo'' perch\'e in realt\`a la mappa di successioni spettrali che induce non permette di usare la funtorialit\`a, ovvero non \`e quella che induce l'iso di s.s.; la naturalit\`a vale per pull-back e per omomorfismi di fibrati sulla stessa base.}
$$
\xymatrix{
H^*(M(\A);\Z) \ar[r]^(.35){\varphi_L^*} \ar[d]_{\overline{\psi}} &  H^*(L;\Z) \otimes H^*(M(\Ar{L});\Z) 
\ar[d]^{\psi_L}\\
\widehat E_2^{p,q} \ar[r]^{\overline{\varphi}_L^*} &  _LE_2^{p,q}.
}
$$
Note that it follows from Corollary \ref{cor:decompos_ss} that the horizontal maps are surjective.
If we take the sum for on the right hand side all layers $L$ we get the commutative diagram
\begin{equation}\label{eq:badass}
\begin{split}
\xymatrix{
H^*(M(\A);\Z) \ar[r]^(.35){\Phi^*} \ar[d]_{\overline{\psi}} & \underset{L\in \Cc}{\bigoplus} H^*(L;\Z) \otimes H^*(M(\Ar{L});\Z) \ar[d]^{\underset{L\in \Cc}{\bigoplus} \psi_L}\\
\widehat E_2^{p,q} \ar[r]^{\overline{\Phi}^*} & \underset{L\in \Cc}{\bigoplus} {}_LE_2^{p,q}.
}
\end{split}
\end{equation}
As above, the vertical maps are group isomorphisms. Moreover the horizontal maps and ${\bigoplus_{L\in \Cc}} \psi_L$ are ring homomorphisms.
We want to describe the image of the subalgebra $\Ab(\A) \subset {\bigoplus}_{L\in \Cc} H^*(L;\Z) \otimes H^*(M(\Ar{L});\Z)$ of coherent elements in the subalgebra of ${\bigoplus}_{L\in \Cc}{}_LE_2^{p,q}.$

Let $L' \in \Cc_q$ be a layer and take a class $\alpha \in \Lc{L'}^q$.
Moreover we assume that $\alpha = \omega \otimes \lambda \in H^*(L';\Z) \otimes H^q(M(\Ar{L'});\Z)$.
It is clear that $\Ab(\A)$ is generated by the elements of the form $p(\alpha)$ (see Definition \ref{df:map_p}):
$$
p(\alpha)_{L} =\left\lbrace
\begin{array}{cl}
\Lc{L' \leq L}(\alpha)=i^*(\omega) \otimes b(\lambda) & \mbox{if }L' \leq L, \\
0 & \mbox{otherwise.}
\end{array}
\right.
$$
Since the map $\bigoplus_{L \in \Cc}\psi_L$ is the identity, it follows from Corollaries \ref{cor:decomp1} and Remark \ref{rmk:decomp2} that the class $\bigoplus_{L \in \Cc}\psi_Lp(\alpha)$ is the image, via $\overline{\Phi}^*$, of the class
$\omega \otimes \lambda \in \widehat{E_2^{p,q}}$. Hence $\Phi^*$ is surjective on the algebra $\Ab(\A)$.

The injectivity of the map $\Phi^*:H^*(M(\A);\Z) \to \bigoplus_{L \in \Cc} \Lc{L}$ follows by a rank argument. In fact both terms are torsion-free and we already know that the Poincar\'e polynomial of
$H^*(M(\A);\Z)$ is given by 
$$
P_{\A}(t) = \sum_{j=0}^{\infty}  |\Nc_j| (1+t)^{d-j}t^j
$$
where $\Nc_j$ is the set of pairs $(L,N) \in \Cc_j \times \mathscr{P}(\A)$ and $N$ is a no broken circuit set of cardinality $j$ 
of $\Cc_{\leq L}.$ We claim that $P_{\A}(t)$ is also the Poincar\'e polynomial of the
algebra $\Ab(\A)$. In fact from Theorem \ref{thm:iso_alg} we have $\Ab(\A) \simeq \Aa(\A)$ and $\Aa(\A)$ is the direct sum of
free modules
$$
\Aa(\A) = \bigoplus_{L\in \Cc} \Lc{L}^{\rk (L)},
$$
and the contribution of the term $\Lc{L}^{\rk (L)}$ for the Poincar\'e polynomial of $\Aa(\A)$ is 
\begin{eqnarray*}
&(1+t)^{d-\rk L}& \rk H^{\rk L}(M(\Ar{L});\Z)=\\=&(1+t)^{d-\rk L}& |\{N \in \mathscr{P}(\A) 
\mbox{ a no broken circuit set in } \Cc_{\leq L}\}|.
\end{eqnarray*}
Taking the sum over all $L \in \Cc$ we get that the two torsion free algebras $H^*(M(\A);\Z)$ and $\Ab(\A)$ have the same Poincar\'e polynomial $P_{\A}(t)$. Hence the surjective map
$\Phi^*:H^*(M(\A);\Z) \to \Ab(\A)$ is also injective.
The theorem follows.
\end{proof}

\section{The general case}\label{sec:general}

From now on we drop the restriction to complexified arrangements and
treat general complex toric arrangements. 
We will show that the description of the cohomology ring of the
complement naturally applies also to this case. In fact, a
deletion-restriction argument allows us to always reduce to the
complexified case.

\subsection{Deletion-restriction recursion}
We thus start with a brief discussion of the effect on the cohomology of removing an
hypertorus from the arrangement and of restricting the arrangement to an hypertorus.

This type of operation has been investigated by Bibby in \cite{bibby2013} and by
Deshpande and Sutar in \cite{Deshpande}. Here we discuss how some of
their results generalize to cohomology with integer coefficients, and
start with a remark on degeneration of spectral sequences.

\begin{rem} \label{rem:notorsion}

In analogy with Theorem \ref{thm:collapses}, the Leray spectral sequence induced by the inclusion $M(\A) \into T$, also considered in 
\cite{bibby2013}
(see also \cite[sec. 4.3]{dupont2013}, gives, as a second term
$$
E_2^{p,q} = \bigoplus_{L \in \Cc_q} H^p(L;\Z) \otimes H^q(M(\Ar{L});\Z) \Longrightarrow H^{p+q}(M(\A);\Z).
$$
Looking at the Poincar\'e polynomial $P_\A(t)$ of $H^*(M(\A);\C)$ as in the proof of Theorem \ref{thm:collapses} in the complexified case, we can see that $E_2$ is a free $\Z$-module and the rank of $E_2$ is the same as the rank of $H^*(M(\A);\C).$ This implies that the spectral sequence collapses at the $E_2$ term and hence the cohomology $H^*(M(\A);\Z)$ is torsion free\footnote{We thank Cl\'ement Dupont for a useful conversation where we noticed this natural generalization of Theorem \ref{thm:collapses}.}.
In particular, since the $E_2$ term of the spectral sequence is isomorphic as a $\Z$-module to the cohomology $H^*(M(\A);\C),$ given a layer $L \in \Cc$ and a class $\alpha \in H^p(L;\Z) \otimes H^{\rk(L)}(M(\Ar{L});\Z)$ we can associate in a natural way an element $\overline{\alpha} \in H^*(M(\A);\C).$
\end{rem}

\begin{lem} \label{lem:inclusion}
Let $\A \subset \B$ be toric arrangements in $T$, the inclusion $M(\B) \subset M(\A)$ in\-du\-ces an injective homomorphism of cohomology rings $i^*: H^*(M(\A);\Z) \into H^*(M(\B);\Z).$
\end{lem}
\begin{proof}
This is straightforward from the description of the Leray spectral sequences $_{\A}E_r^{p,q}$ and $_{\B}E_r^{p,q}$ associated to the inclusion $M(\A) \into T$ and $M(\B) \into T.$
Following the construction of the spectral sequence one can see that the inclusion $i:M(\B) \subset M(\A)$ induces a map of spectral sequences that on the $E_2$-term is as follows:
$$
i_*:\bigoplus_{L \in \Cc_q(\A)} H^p(L;\Z) \otimes H^q(M(\Ar{L});\Z) \to \bigoplus_{L \in \Cc_q(\B)} H^p(L;\Z) \otimes H^q(M(\B[L]);\Z)
$$
where 
the map $i^*$ is given by the sum of the homomorphisms on the summands
$$i^*_L:H^p(L;\Z) \otimes H^q(M(\Ar{L});\Z)  \to H^p(L;\Z) \otimes H^q(M(\B[L])\Z)$$
and $i^*_L$ is given by the identity on the first factor and by the natural injection
$$H^q(M(\Ar{L});\Z) \into H^q(M(\B[L]);\Z)$$
on the second factor induced by the inclusion $M(\B[L]) \into M(\Ar{L})$. The Lemma follows since the spectral sequence collapses at the page $E_2$.
\end{proof}

Given a toric arrangement $\A$, choose $Y_0 \in \A$ and let $\A' = \A \setminus \{ Y_0 \}$ and $\A'' = \{ Y_0 \cap Y \mid \mbox{  for } Y \in \A'\}$. 
We consider $\A'$ as a toric arrangement in $T$, even if its rank differs from the rank of $\A$. We consider $\A''$ as an arrangement in $Y_0$.

\begin{thm} We have the following short exact sequence of groups:
$$
0 \to H^*(M(\A');\Z)\to H^*(M(\A);\Z) \to H^{*-1}(M(\A'');\Z) \to 0.
$$
\end{thm}
\begin{proof}

We observe that, in light of Remark \ref{rem:notorsion}, 
for every toric arrangement $\B$ we have an isomorphism of 
$\Z$-modules 
$$
_{\B}E_2^{p,q} \to H^*(M(\B);\Z),
$$
where we write $_{\B}E_2^{p,q}$ for the Leray spectral sequence 
in the case of the arrangement $\B$.
The result follows applying the isomorphism above to the exact sequence (see \cite{bibby2013})
$$
0 \to {}_{\A'}E_2^{p,q} \to {}_{\A}E_2^{p,q} \to {}_{\A''}E_2^{p,q} \to 0.
$$
\end{proof}

\subsection{The cohomology ring of (non complexified) toric arrangements}

Given a layer $L \in \Cc = \Cc(\A)$ we define, as we did for hyperplane arrangements, the subarrangement:
$$
\A_L:= \{Y \in \A \mid L \in Y\}.
$$

\begin{rem}
For any layer $L \in \Cc(\A)$, the subarrangement $\A_L$ is ``almost'' a complexified toric arrangement, in the following sense: choose an element $p \in L$ and consider the inverse $p^{-1} \in T$ of $p$. Then the translated arrangement $p^{-1}\A_L = \{ p^{-1}Y \mid Y \in \A_L \}$ is complexified.

It is straightforward to see that the descriptions of the cohomology ring of $M(\A)$ given in Theorem \ref{thm:main} and Theorem \ref{thm:iso_alg} extend to this case.
\end{rem}

Let $\Cc_{\max}:= \Cc_{\rk(\A)}$ be the set layers of maximal rank of in $\Cc.$ We recall that if $\A$ is essential, then the elements in $\Cc_{\max}$ are points. Otherwise, given the subtorus $\overline{L}$ that is the translation of any of the element in $\Cc_{\max}$ passing through the identity, we can factor $M(\A) = \overline{L} \times M(\A/\overline{L}),$ where $\A/\overline{L}$ is the essential arrangement induced by $\A$ in $T/\overline{L}.$

For every layer $L \in \Cc = \Cc(\A)$ we can choose a layer of maximal rank $P(L) \in \Cc_{\max}$ contained in $L.$ 
Now, we fix a layer $P_1 \in \Cc_{\max}$ and let $L \in \Cc(\A_{P_1}).$ Let $\alpha \in H^*(L;\Z) \otimes H^{\rk(L)}(M(\A_{P_1}[L]);\Z).$ Moreover, let $\overline{\alpha}$ be the class
corresponding to $\alpha$ in the ring $A(\A_{P_1}) = H^*(M(\A_{P_1});\Z).$ 
We can consider the class 
$$
\beta \in H^*(L;\Z) \otimes H^{\rk(L)}(M(\A_{P(L)}[L]);\Z) 
$$ 
induced by $\alpha$ via the inclusion $i:L \times M(\A_{P(L)}[L]) \into L \times M(\A_{P_1}[L]):$
$$
\beta:= i^* \alpha
$$
and let $\overline{\beta}$ be the corresponding class in $A(\A_{P(L)}) = H^*(M(\A_{P(L)});\Z).$
Finally, let $\widetilde{\alpha}$ (resp. $\widetilde{\beta}$) the class induced by $\overline{\alpha}$ (resp. $\overline{\beta}$) in the tensor product $\bigotimes_{P \in \Cc_{\max}} A(\A_P)$ where the $P_1$-factor (resp. $P(L)$-factor) is $\overline{\alpha}$ (resp. $\overline{\beta}$) and all the other factors equal $1$.

We define the ideal $$I(\A) \subset \bigotimes_{P \in \Cc_{\max}} A(\A_P)$$ generated 
by the elements of the form 
$$\widetilde{\alpha} - \widetilde{\beta}$$ for any pair $(\widetilde{\alpha}, \widetilde{\beta})$ constructed as above. Moreover we define the ideal $$J(\A)\subset \bigotimes_{P \in \Cc_{\max}} A(\A_P)$$  generated by all the cup products of the form
$$
\tal_1 \cup \cdots \cup \tal_h
$$
for some disjoint layers $P_{j_1}, \ldots, P_{j_h} \in \Cc_{\max}$ and $L_{j_1}, \ldots, L_{j_h} \in \Cc$ such that:
\begin{itemize}
\item[i)] $P(L_{j_i}) = P_{j_i}$;
\item[ii)] $\tal_i$ ($i = 1, \ldots, h$) is an element in the tensor product $\bigotimes_{P \in \Cc_{\max}} A(\A_P)$ induced 
by $\oal_i \in H^*(M(\A_{P_{j_i}});\Z)$ on the $P_{j_i}$-factor and 
all the other factors equal $1$, with $\oal_i$ induced by 
a class $\alpha_i \in H^*(L_{j_i};\Z) \otimes H^{\rk(L_{j_i})}(M(\A_{P{j_i}}[L_{j_i}]);\Z)$; 
\item[iii)] the layers $L_{j_i}$ have trivial intersection: $$\bigcap_{i = 1}^h L_{j_i} = \emptyset.$$
\end{itemize}

We consider the map 
$$
\Delta: M(\A) \to \prod_{P \in \Cc_{\max}} M(\A_P)
$$
induced by the inclusions $M(\A) \into M(\A_P)$.

The next proposition is useful in order to understand the corresponding cohomology homomorphism
$$
\Delta^*: \bigotimes_{P \in \Cc_{\max}} H^*(M(\A_P);\Z) \to H^*(M(\A);\Z).
$$
\begin{prop}
The homomorphism $\Delta^*$ is surjective and the kernel of $\Delta^*$ is given by the ideal $I(\A)+J(\A)$.
\end{prop} 
\begin{proof}
We begin showing that the map $\Delta^*$ is surjective. Let consider a layer $L \in \Cc$ and and an element $\alpha \in H^*(L;\Z) \otimes H^{\rk(L)}(M(\Ar{L});\Z).$  Let $\overline{\alpha} \in  H^*(M(\A);\Z)$ be the corresponding class. The hyperplane arrangements $\Ar{L}$ and $\A_{P(L)}[L]$ are equal. Hence we can consider the class $\overline{\beta} \in H^*(M(\A_{P(L)});\Z)$ associated to $$\alpha \in H^*(L;\Z) \otimes H^{\rk(L)}(M(\A_{P(L)}[L]);\Z) = H^*(L;\Z) \otimes H^{\rk(L)}(M(\Ar{L});\Z).$$
From the description of the map $i^*: H^*(M(\A_{P(L)});\Z) \to H^*(M(\A);\Z)$ given in the proof of Lemma \ref{lem:inclusion} we have that $i^*(\overline{\beta}) = \overline{\alpha}$. Hence we can consider the class 
$$\widetilde{\beta} \in \bigotimes_{P \in \Cc_{\max}} H^*(M(\A_P);\Z) = \bigotimes_{P \in \Cc_{\max}}  A(\A_P)$$ 
given by the product of the term $\overline{\beta}$ in the $P(L)$-factor and $1$ for all other factors and we have
\begin{equation} \label{eq:Deltastar}
\Delta^*\widetilde{\beta} = \overline{\alpha}
\end{equation}
and the surjectivity of $\Delta^*$ follows since the element $\overline{\alpha}$ runs over a set of generators of $H^*(M(\A);\Z)$.

Let us define the $\Z$-submodule $$V(\A) \subset \bigotimes_{P \in \Cc_{\max}} A(\A_P)$$ generated by all the classes $\widetilde{\beta}$ as above, that is with $\widetilde{\beta}$ induced in one of the factors of the tensor product by a class $\beta \in H^*(L;\Z) \otimes H^{\rk(L)}(M(\A_{P(L)}[L]);\Z)$ for all possible layers $L \in \Cc(\A)$ and all the other factors equal to $1$.
We notice that by equation \eqref{eq:Deltastar} the restriction of the map $\Delta^*$ to $V(\A)$ is injective.

It is clear that the ideal $I(\A)$ is contained in the kernel of $\Delta^*.$ We need to show that also $J(\A) \subset \ker\Delta^*$. Given a generator $\tal_1 \cup \cdots \cup \tal_h,$
we can consider the images $\Delta^*\tal_i$. From the Leray spectral sequence associated to the inclusion $M(\A) \into T$ it follows that $\Delta^*\tal_i$ can be represented by a cocycle supported in a neighborhood of $L_{j_i}$. Since we can choose neighborhoods $U_1, \ldots, U_h$ of the layers $L_{j_1}, \ldots, L_{j_h}$ such that $\bigcap U_i = \emptyset$, this implies that the product $\Delta^*\tal_1 \cup \cdots \cup \Delta^*\tal_h$ must be trivial.

In order to show that the kernel of $\Delta^*$ is the ideal $I(\A)+J(\A)$ it remains to show that any element of $\bigotimes_{P \in \Cc_{\max}} H^*(M(\A_P);\Z)$ is equivalent, modulo $I(\A)+J(\A)$, to an element in the submodule $V(\A).$

Let $\omega$ be an element in $\bigotimes_{P \in \Cc_{\max}} H^*(M(\A_P);\Z) = \bigotimes_{P \in \Cc_{\max}} A(\A_{P}).$ We can reduce to the case of $\omega = \bigotimes_{P \in \Cc_{\max}} \oal_P$, with $\oal_{P} \in A(\A_{P}) $. 

If we write $\tal_P$ for the tensor product with $P$-factor $\oal_P \in A(\A_P)$ and $1$ for all other factors, we have 
$$
\omega = \tal_{P_1} \cup \cdots \cup \tal_{P_k}
$$
where $\Cc_{\max} = \{P_1, \ldots, P_k \}$.

Moreover we can suppose that for every $P \in \Cc_{\max}$ the class $\oal_P$ is induced by a class 
$\alpha_P \in H^*(L_P) \otimes H^{\rk(L_P)}(M(\A_P[L_P]);\Z)$ for some $L_P \in \Cc(\A_P)$.

If $\bigcap_{P \in \Cc_{\max}} L_P = \emptyset$ then $\omega \in J(\A)$ and hence $\omega \equiv 0 \mod I(\A) + J(\A)$.

Otherwise, the intersection of the layers $L_P$ is non-empty. Let $\oP \in \Cc_{\max}$ be such that $\oP \in \bigcap_{P \in \Cc_{\max}} L_P.$

Since $\oP \subset L_P$ for all $P \in \Cc_{\max}$, we have that the local arrangements 
$\A_P[L_P]$ and $\A_{\oP}[L_P]$ are equal. Hence the class $\tal_{P}$ is equivalent, modulo the ideal $I(\A)$, to the class $\tga_P$ that is the tensor product with $\oP$-factor $\oga_P \in A(\A_{\oP})$ and all the other factors equal $1$ and $\oga_P$ is induced
by the class $\alpha_P \in  H^*(L_P) \otimes H^{\rk(L_p)}(M(\A_{\oP}[L_P]);\Z) =  H^*(L_P) \otimes H^{\rk(L_p)}(M(\A_P[L_P]);\Z)$. Hence we have reduced, modulo $I(\A)$, the class $\omega$ to the product of the classes $\tga_P$:
$$
\omega \equiv \tga_{P_1} \cup \cdots \cup \tga_{P_k} \mod I(\A)
$$
and the right hand side is a tensor product with $\oP$-factor equal to $\oga_{P_1} \cup \cdots \cup \oga_{P_k} \in A(\A_{\oP})$ and all other factors equal to $1$.
Finally, since $A(\A_{\oP})$ decomposes as a direct sum
$$\bigoplus_{L \in \Cc(\A_{\oP})} H^*(L) \otimes H^{\rk(L)}(M(\A_{\oP}[L]);\Z)$$ 
we have that $$\oga_{P_1} \cup \cdots \cup \oga_{P_k} = \sum_{L \in \Cc(\A_{\oP})} \ode_L$$ 
where $\ode_L$ is induced by a class in $H^*(L) \otimes H^{\rk(L)}(M(\A_{\oP}[L]);\Z)$. So we can write 
$$\omega \equiv \sum_{L \in \Cc(\A_{\oP})} \tde_L$$ where $\tde_L$ is the class in tensor product with $\oP$-factor $\ode_L$ and all other factors equal to $1$. Now it is clear that each of the summands $\tde_L$ can be replaced, modulo $I(\A)$ with an element in $V(\A)$ and this complete the proof.
\end{proof}

We recall that in Definition \ref{def:algebra_a} we introduced the ring 
$$
A(\A) = \bigoplus_{L\in \Cc} \Lc{L}^{\rk (L)}.
$$

\begin{rem}
We notice that the definitions of the algebras $A(\A)$ and $B(\A)$ (Definition \ref{def:algebra_b}) do not depend on the structure of complexified arrangement. In particular all the results in Section \ref{ss:algebras} hold for any arrangement and  Proposition \ref{prop:iso_algebras} gives, for any arrangement $\A$, the isomorphism
$$
A(\A) \simeq B(\A).
$$
\end{rem}

We can then state and prove in full generality the result of Theorem \ref{thm:main}:

\begin{thm} \label{thm:final_part_A}
There exists a well defined map $\overline{\iota}:\left(\bigotimes_{P \in \Cc_{\max}} A(\A_P)\right) / \ker \Delta^* \to  A(\A) $ that induces the isomorphism
$H^*(M(\A);\Z) \to A(\A).$
\end{thm}
\begin{proof}
For any $P \in \Cc_{\max}$ there's a natural map $\iota_P: A(\A_P) \to A(\A)$ that is induced on the summands by the maps
$$
\Lc{L}^{\rk (L)}(\A_P) \to \Lc{L}^{\rk (L)}(\A)
$$ induced by the identifications $M(\A[L]) \subset M(\A_P[L])$. 
It is easy to see that the map $\iota_P$ is a ring homomorphism and is injective. Hence there is a well defined map
$$
\iota:\bigotimes_{P \in \Cc_{\max}} A(\A_P) \to  A(\A) 
$$
defined by the cup product of the generators:
$$
\iota (a_{P_1} \otimes \cdots \otimes a_{P_k}) = \iota_{P_1}(a_{P_1}) \cup \cdots \cup \iota_{P_k}(a_{P_k}).
$$
We will show that the kernel of the map $\iota$ is $\ker \Delta^*.$ Given classes $a_i \in \Lc{L_i}^{\rk (L_i)}(\A)$, for $i=1, \ldots, h$ it is clear that $a_1 \cup \cdots \cup a_h$ is supported only on the summands $\Lc{L}^{\rk (L)}(\A)$ such that $L \subset \bigcap_i L_i$ and this implies that $J(\A) \subset \ker \iota$.

Moreover, given a maximal layer $P_1 \in \Cc_{\max}(\A)$ and a layer $L \in \Cc(\A_{P_1})$, we have that $\A_{P_1}[L] = \A[L] = \A_{P(L)}[L]$. Hence there is a commutative diagram
$$\xymatrix{
\Lc{L}^{\rk (L)}(\A_{P_1})\ar[rr]^{=} \ar[dr]^{=} & & \Lc{L}^{\rk (L)}(\A) \\
& \Lc{L}^{\rk (L)}(\A_{P_L}) \ar[ur]^{=} & 
}
$$
and this implies that $I(\A) \subset \ker \iota$ and the map $\iota$ is surjective.
Then the map $\iota$ induces a well defined surjective map $\overline{\iota}$. 
The injectivity of $\overline{\iota}$ follows from a rank-counting argument, since the $\Z$-modules
$A(\A)$ and $\bigotimes_{P \in \Cc_{\max}} A(\A_P)$ have the same rank given by the sum
$$
\sum_{L \in \Cc(\A)} 2^{\rk(L)} \dim H^{\rk(L)}(M(\A[L]);\Q).
$$
\end{proof}

\section{Dependency on the poset of layers}\label{sec:combiforse}

We now turn to the study of the relationship between the cohomology algebra structure and the poset of layers. We refer to Section \ref{sec:combin_aspects} for a discussion of the our results on this topic.

\subsection{Whitney homology of the poset of
  layers}\label{sec:Whitney} In analogy with the case of hyperplanes,
the additive structure of the toric OS-algebra can be obtained as the
(Whitney) homology of the sheaf of rings $\mathscr W: L \mapsto
H^*(L;\Z)$ defined on the poset of layers $\Cc$ with zero maps as
restrictions. In fact, in this situation the differentials of the
associated spectral sequence \cite[Section 4]{baclawski}
vanish already at the first page, thus we obtain 
$$
H^q(\Cc,\mathscr W) = \bigoplus_{L\in \Cc, \rk(L) = q} H^*(L;\Z )^{\mu(T^d,L)}\simeq \bigoplus_{L\in \Cc, \rk(L) = q} H^*(L;\Z )\otimes \Z^{\mu(T^d,L)} \simeq $$
$$ \bigoplus_{L\in \Cc, \rk(L) = q} H^*(L;\Z )\otimes H^q(M(\A[L]);\Z)
$$
since the absolute value of the M\"obius function at $L$ is precisely the number of maximal no broken circuit sets of $\Cc_{\leq L}$.

\subsection{Centered arrangements with unimodular basis}

Suppose that the given toric arrangement is centered, i.e., each subtorus is of the form $Y_i=\ker \chi_i$ (for notations see Section \ref{sec:maindef}). We can identify the lattice $\operatorname{Hom} (T,\C^*)$ with $\Z^d$ through any isomorphism.  The characters then correspond to the columns of a $(d\times n)$ matrix $A$ with integer entries.   For any subset $I\subseteq [n]$ let then $A(I)$ denote the matrix given by the columns of $A$ with indices in $I$ and let $m(I)$ denote the product of the invariant factors of $A(I)$. Then the function $m(\cdot)$ defines an arithmetic matroid on the $\mathbb Q$-dependency matroid of the columns of $A$ (we refer to \cite{BrMo} for basics on arithmetic matroids); the matrix A is then called a {\em representation} of this arithmetic matroid.
In particular notice that, if $\vert I \vert = d$, we have that
$m(I)=\vert \det A(I) \vert$.

Now suppose that the set of defining characters has a unimodular basis, i.e., that - say - $\chi_1,\ldots ,\chi_d$ are a basis of the lattice 
$\operatorname{Hom}(T,\C^*)$. Notice that the existence of a unimodular basis can be ascertained from the multiplicity data of the associated arithmetic matroid: such a basis has multiplicity $1$. We choose the isomorphism $\Hom(T,\mathbb C^*) \simeq \mathbb Z^d$ to be such that $\chi_i$ is sent to the standard vector $e_i$ for $i\leq d$. Then, the leftmost $d\times d$-block of $A$, $A(\{1,\ldots ,d\})$, is the identity matrix.

We now claim that in this case the whole matrix $A$ can be recovered from the multiplicity data (which is part of the information given by the poset of layers).

\begin{thm}\label{thm:ricostruzione}
  If an arithmetic matroid with a basis of multiplicity $1$ is representable by a matrix $A$, then $A$ is unique up to sign reversal of the column vectors.
\end{thm}

\begin{proof}

For every non-zero entry $a_{i,j}$ of $A$, ($j>d$), we have $a_{i,j} = (-1)^i\det(A(1,\ldots, i-1,j,i+1,\ldots,d))$.
$$
a_{i,j}=m([d]\setminus \{i\} \cup \{j\}) \operatorname{sign}(\det A (j,1,\ldots,i-1,i+1,\ldots,d)).
$$

We may without loss of generality (by taking negatives of characters) suppose that the first nonzero entries past the first $d$ columns in every row are positive, i.e., for $$j_0(i):=\min\{j>d \mid a_{i,j} \neq 0\},$$ 
$$a_{i,j_0} >0$$
and that the first entry in every column is positive, i.e., if $$i_0(j):=\min\{i\mid a_{i,j}\neq 0\},$$ 
$$a_{i_0(j),j} = \det(A(j, 1, \cdots, \widecheck{i_0(j)}, \cdots, d)) >0.$$

We will determine the sign of the entries $a_{i,j}$ by induction on $j$. We assume $j \geq d$ and if $j=d$ there's nothing to prove. 

For a fixed $j>d$ we will determine the sign of the entries $a_{i,j}$ by induction on $i$.
For fixed $i$ and $j$, we restrict to study the submatrix $A'$ of $A$ given by the columns $1, \ldots, d$ and $j' >d$ such that $a_{i,j'} \neq 0$ of $A$ and to the corresponding submatroid $m'$.
Then we can assume that $a_{i,j'} \neq 0$ for all $d < j' \leq j$.

Let $i_0(j)$ be as above, we recursively define the integers $i_h(j)$ for $h>0$ as follows:
$$
i_h(j):=\min\{i >i_{h-1}(j)\mid a_{i,j}\neq 0\}
$$
and let $$\overline{k}(j):= \min\{k \mid j_0(i_k(j)) \neq j\}.$$

We notice that if $\overline{k}(j) = \infty$ then $j_0(i_k(j))=j$ for all $k$ and hence we are assuming that $a_{i,j} \geq 0$ for all $i$. In general, since for $k \leq \overline{k}(j)$ we have that 
$j_0(i_k(j)) = j$,
we can assume that $a_{i_k(j),j} >0$ for all $k < \overline{k}(j)$.

Assume $\overline{k}(j) < \infty$.
Hence for a fixed $i$, if we suppose $i>i_0(j)$ and $a_{i,j}\neq 0$, we have three possible cases:
\begin{itemize}
\item[a)] if $i_{\overline{k}(j)}(j)>i$ then we have $i = i_k(j)$ for a certain $k \leq \overline{k}(j)$  and hence we are already assuming that $a_{i,j}>0$;

\item[b)] if $i_{\overline{k}(j)}(j)=i$ then, up to changing the sign of the $i_0(j), \ldots, i_{\overline{k}(j)-1}(j) $-th rows of $A$ and of the $i_0(j), \ldots, i_{\overline{k}(j)-1}(j) $-th and $j$-th columns of $A$, we can assume that $a_{i,j}>0$;

\item[c)] if $i_{\overline{k}(j)}(j)<i$ then we already know by induction the sign of 
$a_{i,j_0(i_{\overline{k}(j)}(j))}$ 
and since from the hypothesis of restriction we have that $a_{i,j_0(i_{\overline{k}(j)}(j))}\neq 0$, we can consider, up to a sign, the determinant
$$
| \det A'(j_0(i_{\overline{k}(j)}(j)),j, 1, \cdots ,\widecheck{{i_{\overline{k}(j)}(j)}}, \cdots, \widecheck{i}, \cdots, d) |= $$
$$
=| a_{i,j} a_{i_{\overline{k}(j)}(j), j_0(i_{\overline{k}(j)}(j))} - a_{i,j_0(i_{\overline{k}(j)}(j))} a_{i_{\overline{k}(j)}(j),j} |.
$$
All entries in the right hand side are non zero and we have
$$
m'([d]\setminus \{i_{\overline{k}(j)}(j), i \} \cup \{j_0(i_{\overline{k}(j)}(j)),j\})= 
| \det A'(j_0(i_{\overline{k}(j)}(j)),j, 1, \cdots ,\widecheck{{i_{\overline{k}(j)}(j)}}, \cdots, \widecheck{i}, \cdots, d) |
$$
we can determine the sign of $a_{i,j}$.
\end{itemize}
\vspace{-5mm}
\end{proof}

\subsection{Further questions and examples}\label{sec:examples}

We close by presenting some examples addressing the dependency of
the ring structure from the arrangement's combinatorics.

\subsubsection{Isomorphism type}

Since our description of the cohomology ring of the complement of a toric arrangement depends on the defining equation of the arrangement, we are led to consider the following problem:
\begin{question}
Is the toric OS-algebra combinatorial? Does the ring $H^*(M(\A);\Z)$ only depend on the poset $\Cc(\A)$?
\end{question}

We provide an example which shows the delicacy of the situation, even in small rank. We give two complexified toric arrangements (of rank $2$) with isomorphic posets of layers whose integer cohomology rings are indeed isomorphic -- yet the isomorphism can't be chosen to be natural with respect to the inclusion into the ambient torus $(\C^*)^2$.

\begin{es} \label{ex:noncombinatorial}
Fix $d=2$ and let $T = (\C^*)^2$ the $2$-dimensional complex torus with coordinates $z_1, z_2.$ 
We write $H_{ij}$ for the subtorus defined by the equation
$H_{ij}=\{(z_1,z_2) \in T \mid z_1^iz_2^j = 1\}.$ Moreover we write $\chi_{ij}$ for the character defined by $\chi_{ij}(z):= z_1^iz_2^j$. Consider the arrangements $\A^1, \A^2$ defined as follows:
$$
\A^1= \{H_{10},H_{15}\}
$$
and
$$
\A^2= \{H_{10},H_{25} \}.
$$
The two arrangements have isomorphic posets of layers described in Figure \ref{layers_ex}.
\begin{figure}[htb] 
\resizebox{.3\textwidth}{!}{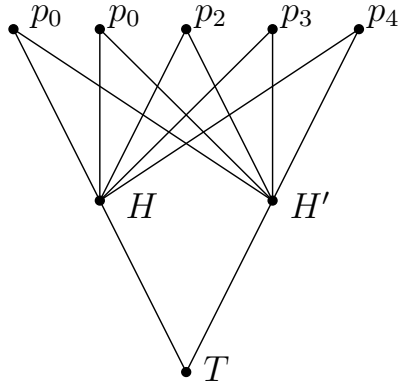}
\caption{The poset of layers associated to the arrangements $\A^1$ and $\A^2$ of Example \ref{ex:noncombinatorial}.}
\label{layers_ex}
\end{figure}

Let $x$ be a generator of $H^1(\C^*; \Z)$ and $y$ a generator of $H^1(\C \setminus \{1\}; \Z).$

The group $H^1(M(\A^1);\Z)$ has rank $4$ and is generated by the classes $x_1:= \chi_{10}^*(x), x_2:=\chi_{01}^*(x), 
y_1:=\chi_{10}^*(y), y_2:=\chi_{15}^*(y).$ The inclusions $i_1: \C^*\setminus\{1\} \into \C^*$ and
$i_2: \C^*\setminus\{1\} \into \C \setminus \{1\}$ give the relation 
\begin{equation}\label{eq:rel_xy}
i_1^*(x) i_2^*(y) = 0.
\end{equation}
From the identity $\chi_{15}= \chi_{10}\chi_{01}^5$ we
get $\chi_{15}^*(x)=x_1+5x_2$ and hence, applying Equation \eqref{eq:rel_xy}, we get the following relations for the ring $H^*(M(\A^1);\Z)$:
\begin{equation}\label{eq_1a}
x_1y_1=0
\end{equation}
and
\begin{equation}\label{eq_1b}
(x_1+5x_2)y_2=0.
\end{equation} 
Moreover we have the square relations $x_1^2 = x_2^2 = y_1^2 = y_2^2 =0.$
Let $p^1_0, \ldots, p^1_4$ the points of the intersection $H_{10}\cap H_{15}$, the group $H^2(M(\A^1);\Z)$
has rank $8$ and is generated by $x_1x_2, x_2y_1, x_1y_2, \tau_0, \ldots, \tau_4$, where $\tau_k$ corresponds to the class $1 \otimes \alpha_k$ in the algebra $A(\A^1)$ and $\alpha_k$ is a top class generating the group $H^2(M(\A^1[p_k^1]);\Z).$ The ring structure is completed by the relation
$$
y_1y_2 = \sum_{k=0}^4 \tau_k.
$$

The analogous computation for $\A^2$ goes as follows. The group $H^1(M(\A^2);\Z)$ has rank $4$ and is generated by the classes $x_1:= \chi_{10}^*(x), x_2:=\chi_{01}^*(x), 
y_1:=\chi_{10}^*(y), y_2':=\chi_{25}^*(y).$ 
The following relations holds for the ring $H^*(M(\A^2);\Z)$:
\begin{equation}\label{eq_2a}
x_1y_1=0
\end{equation}
and
\begin{equation}\label{eq_2b}
(2x_1+5x_2)y'_2=0.
\end{equation} 
Moreover we have the square relations $x_1^2 = x_2^2 = y_1^2 = {y_2'}^2 =0.$
Let $p^2_0, \ldots, p^2_4$ the points of the intersection $H_{10}\cap H_{25}$, the group $H^2(M(\A^2);\Z)$
has rank $8$ and is generated by $x_1x_2, x_2y'_1, (x_2-2x_1)y'_2, \tau_0', \ldots, \tau_4'$, where $\tau_k'$ corresponds to the class $1 \otimes \alpha'_k$ in the algebra $A(\A^2)$ and $\alpha'_k$ it a top class generating the group $H^2(M(\A^2[p_k^2]);\Z).$ The ring structure is completed by the relation
$$
y_1y_2' = \sum_{k=0}^4 \tau_k'.
$$

With an easy computation one can see that the annihilator of an element $u\neq 0$ of dimension $1$ in the ring $R_1 = H^*(M(\A^1);\Z)$ (that we simply call $\mathrm{Ann}_1^1(u)$) is a subgroup of $R_1^1$ that has rank $1$, except when $u$ belongs to one of the following two groups:
$$
G_1=\{ a x_1 + b y_1 | a,b \in \Z \}
$$
or 
$$
G_2= \{a (x_1+5x_2) + b y_2 | a,b \in \Z\}.
$$
In such cases the rank of $\mathrm{Ann}_1^1(u)$ is $2.$

Similarly, for $u\neq 0$ of dimension $1$ in the ring $R_2=H^*(M(\A^2);\Z)$, the rank of $\mathrm{Ann}_1^1(u)$ has rank $2$ if and only if $u$ belongs to one of the following two groups:
$$
G_1=\{ a x_1 + b y_1 | a,b \in \Z \}
$$
or 
$$
G_2'=\{ a (2x_1+5x_2) + b y_2' | a,b \in \Z\}.
$$

It is easy to verify that the map $f:R_1 \to R_2$ defined as follows
$$ \begin{array}{ll}
 f:x_1 \mapsto  2x_1+5y_1;& f:y_1 \mapsto x_1 + 2 y_1; \\
 f:x_2 \mapsto  x_2 - y_1;& f: y_2 \mapsto y_2';  \\
 f:\sum_{i=0}^4 \tau_i \mapsto  x_1y_2' + 2 \sum_{i=0}^4 \tau_i'; &  f:\tau_i - \tau_j \mapsto \tau_i' - \tau_j'
 \end{array}$$
is an isomorphism of rings.

In order to show the impossibility of a natural isomorphism we can consider the ring $R_0=H^*(T; \Z) = \Lambda[x_1, x_2]$. 
The inclusion of $M(\A^i)$ ($i=1,2$) into $T$ induces a structure of $R_0$-algebra on $R_i$.

We claim that the pairs of rings $(R_1, R_0)$ and $(R_0,R_2)$ are not isomorphic, hence
the two cohomology ring $R_1$ and $R_2$ are not isomorphic as algebras on the cohomology of $T$.

In fact we can consider the groups $G_1 \cap R_0$ and $G_2 \cap R_0$ for $(R_1, R_0)$ and 
respectively $G_1 \cap R_0$ and $G_2' \cap R_0$ for $(R_2, R_0)$. In the first pair
a sum of generators of the two intersections is a multiple of $5$ (namely $(x_1+5x_2) - x_1 = 5 x_2$) while in the second pair this is not possible since the two intersections are 
generated by $x_1$ and $2x_1+x_5$.
\end{es}

\subsubsection{Degree one generators}

The question of whether a cohomology ring is generated in degree one
is natural and well-studied. For toric arrangements, this question has
been addressed also in \cite{dp2005, Deshpande}. 

\begin{question}
When is the cohomology ring $H^*(M(\A);\Z)$ generated in degree $1?$
Is this property combinatorially determined by $\Cc(\A)$?
\end{question}

\noindent In order to have that the cohomology ring $H^*(M(\A);\Z)$ is generated in degree $1,$ it is natural to ask as a necessary condition ensuring that intersections can distinguish layers, that is
\begin{center}
\emph{for every $k$, the Boolean algebra generated by the non-empty intersections \\of rank $k$ of elements of $\A$ contains all the layers of rank $k.$}
\end{center}
However, the following example shows that this condition is not sufficient.

\begin{es}
Fix $d=2$ and let $T = (\C^*)^2$ the $2$-dimensional complex torus with coordinates $z_1, z_2.$ 
We define the arrangement $\A$ given by the following subtori (see figure \ref{ex_notgen}):
\begin{figure}[htb] 
\resizebox{.4\textwidth}{!}{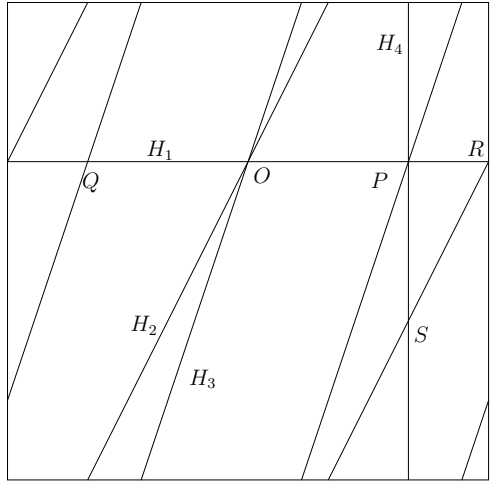}
\caption{Example of a toric arrangement with cohomology ring not generated in degree $1$.}
\label{ex_notgen}
\end{figure}
$$\begin{array}{llll}
H_1= \{z_1 = 1 \}; &H_2= \{z_1z_2^2 = 1 \}; &
H_3= \{z_1z_2^3 = 1 \}; &H_4= \{z_2 =  e^{\frac{2\pi \imath}{3}}\}.
\end{array}$$
It is easy to check that the intersections of rank $2$ are the following:
$$
\{S\}= H_2 \cap H_4; \{P\}= H_1 \cap H_4; \{Q, O, P\} = H_1 \cap H_3; \{R,O\}=H_1 \cap H_2; \{O\} = H_2 \cap H_3 
$$
and hence they generate a Boolean algebra containing all the layers of rank $1$.

Nevertheless, we claim that the algebra $H^*(M(\A);\Z)$ is not generated in rank $1$. In fact the five local arrangements of rank $2$, namely $\A[P], \A[Q], \A[R], \A[O]$ and $\A[S]$, determines a submodule of rank $7$ in $H^2(M(\A);\Z)$. This module can be generated only by products of the generators associated to the four hypertori of $\A$. The claim follows since ${\binom{4}{2}} = 6 < 7 $.
\end{es}

\nocite{*} 
\bibliographystyle{acm}
\bibliography{bibliotos} 

\end{document}